\documentclass[11pt]{amsart}

\usepackage{amssymb, amsmath, amsthm}

 \newtheorem{thm}{Theorem}[section]
 \newtheorem{cor}[thm]{Corollary}
 \newtheorem{lem}[thm]{Lemma}
 \newtheorem{prop}[thm]{Proposition}

    \theoremstyle{definition}
 \newtheorem{defn}[thm]{Definition}

 \newtheorem{exam}[thm]{Example}
 \theoremstyle{remark}
 
\numberwithin{equation}{section}
\newcommand{\ldb}{\mathopen{\lbrack\!\lbrack}}
\newcommand{\rdb}{\mathclose{\rbrack\!\rbrack}}

\begin{document}

\title[BiHom-Poisson conformal superalgebras arising from related algebraic structures]{BiHom-Poisson conformal superalgebras arising from related algebraic structures}

\author{Bakolè LOUAMOU, Ibrahima BAKAYOKO and Kol Béatrice GAMOU}

\address{D\'epartement de Math\'ematiques, Université de N'Z\'er\'ekor\'e, BP : 50, N'Z\'er\'ekor\'e, Guinea}
\email{ibrahimabakayoko27@gmail.com}
\address{D\'epartement de Math\'ematiques, Institut Ouest Africain de Mathématiques, BP : 1147, Conakry, Guinea}
 \email{louamoubakole30@gmail.com}
\address{D\'epartement de Math\'ematiques, Université de N'Z\'er\'ekor\'e, BP : 50, N'Z\'er\'ekor\'e, Guinea}
\email{kolbeatrice18@gmail.com}
% \thanks{}

\subjclass[2020]{}


% \date{September 2026.}

%%%-------------------------------------------------------------------

\begin{abstract}
Inspired from some works on constructions of algebraic structures from other ones, we purpose to give some construction theorems on BiHom-Poisson conformal superalgebras which unify the notion of the $\mathbb{Z}_2$-graded Poisson algebras, the  BiHom-Poisson algebras and their conformal setting. So, we introduce  BiHom-Poisson conformal superalgebras and  show that they are closed under twisting. We then prouve that any regular BiHom-associative conformal algebra leads
to  BiHom-Poisson conformal algebra via the commutator bracket. Then, we point out that any  BiHom-Poisson conformal superalgebra together
with Rota-Baxter operator give rises to another  BiHom-Poisson conformal superalgebra. Next, we show that tensor product of any BiHom-associative superalgebra and any  BiHom-Poisson conformal superalgebra is also a  BiHom-Poisson conformal superalgebra. Contruction from Hom-post-Poisson conformal superalgebras is also given.
\end{abstract}

%%% ----------------------------------------------------------------------
\maketitle
%%% ----------------------------------------------------------------------

\section{Introduction}
A Poisson algebra is both an associative algebra and a Lie algebra together with a compatible relation called Leibniz rule. They are used in many fields in mathematics and physics. In mathematics they play a fundamental role in Poisson geometry, quantum groups, and deformation
of commutative associative algebras. In physics, they are a major part of deformation quantization, Hamiltonian mechanics, and topological field theories. Poisson-like structures are also used in the study of vertex operator algebras.

Hom-algebraic structures were appeared for the first time in the works of Aizawa N. and Sato H. as a generalization of Lie algebras. Interesting Hom-type algebraic structures of many classical structures were studied as G-Hom-associative algebras \cite{MA}, Hom-alternative algebras, Hom-Malcev algebras, Hom-Poisson algebras, constructions of n-Lie algebras, Hom-Nambu-Lie algebras and Hom-Jordan algebras. However, Post-Lie algebras first arise from the work of Bruno Vallettein 2007 through the purely operadic technique of Koszul dualization.

In 2015, G. Graziani, A. Makhlouf, C. Menini and F. Panaite introduced BiHom-algebraic structures  from a categorical approach as an extension of the class of Hom-algebras. Since then, other interesting BiHom-type algebraic structures of many Hom-algebraic structures has been intensively studied as BiHom-Lie algebras structures, Representations of BiHom-Lie algebras, BiHom-Leibniz algebras, BiHom pre-Lie algebras  \cite{LL}, infinitesimal Hom-bialgebras and the associative (Bi)Hom-Yang-Baxter equations, The construction and deformation of BiHom-Novikov algebras, On n-ary Generalization of BiHom-Lie algebras and BiHom-associative algebras, Rota-Baxter operators on BiHom-associative algebras and related structures \cite{L}.

BiHom-algebraic structures  have recently become popular, due in part to the prospect of having a
general framework in which one can produce many types of natural deformations of algebras. These algebras has been extended to the graded case, among which one can cite, in the case the abelian group giving the graduation is $\mathbb{Z}_2$,  Hom-Lie superalgebras, Hom-Lie admissible superalgebras, Rota-Baxter operator on pre-Lie superalgebras, Hom-Novikov superalgebras, Hom-alternative superalgebras, Hom-Malcev superalgebras, Hom-Jordan
superalgebras and beyond.

Among other results, H. Adimi and his coauthers introduce and give some construction results of BiHom-post-Lie algebras \cite{H} which are a generalization of both post-Lie-algebras and Hom-post-Lie algebras. Their conformal version are studied in \cite{LY}.
% \medskip
J. Liu, H. Zhou recall the notion of Poisson conformal algebras \cite{J}
and describe some constructions of them. Then, they study the formal distribution and introduce conformal formal deformations of commutative associative conformal algebras and show
that Poisson conformal algebras are the corresponding semi-classical limits. Then, they develop
the cohomology theory of noncommutative Poisson conformal algebras and use this cohomology
to study their deformations.
While
BiHom-
Poisson conformal algebra are introduced by S. Asif, Y. Wang in \cite{S} as generalization of both BiHom-associative conformal algebras and BiHom-Lie conformal algebras \cite{SG, T}.  They give some  constructions of these structures including
direct product and twisting. See \cite{SA} for Hom-Leibniz conformal algebras.

The aim of this paper is to give some constructions of BiHom-Poisson conformal superalgebras.
In section two, we recall Preliminaries and some results BiHom-associative conformal superalgebras, BiHom-Lie conformal superalgebras, BiHom-post-Lie conformal superalgebras and BiHom-tridendriform conformal superalgebras.
Section three contains the mains results, that is constructions of BiHom-Poisson conformal superalgebras from twisting, from regular BiHom-associative conformal algebra, from Rota-Baxter operator, from averaging operators, from tensor product of BiHom-associative conformal superalgebra and  BiHom-Poisson conformal superalgebra and from  Hom-Poisson conformal superalgebras.

\section{Preliminaries and some results}
\subsection{Generalities}
\begin{defn}
 A  BiHom  conformal superalgebra is a quadruple $(A, \circ_\lambda, \varphi, \psi)$, where $A$ is a $\bf Z_2$-graded vector space endowed with  a structure of $\mathbb{C}[\partial]$-module compatible with the parity, $\varphi, \psi : A\rightarrow A$ are two commuting even linear maps, $\circ_\lambda$ is a bilinear map $A\times A\rightarrow A[\lambda], x\otimes y\mapsto x\circ_\lambda y$ map satisfying the sesquilinearity relations :
 \begin{eqnarray}
\partial x\circ_\lambda y=-\lambda x\circ_\lambda y, \;
x\circ_\lambda\partial y=(\partial+\lambda)x\circ_\lambda y, \; \partial\varphi=\varphi\partial,\quad \partial\psi=\psi\partial,\; \forall x, y\in A.
 \end{eqnarray}
 The linear maps $\varphi$ and $\psi$ are called either twisted maps or structural maps.
 \end{defn}

A conformal BiHom-superalgebra is said to be commutative if
\begin{eqnarray}
 x\circ_\lambda y=(-1)^{|x||y|}y\circ_{-\partial-\lambda}x,\quad \forall x, y\in \mathcal{H}(A).
\end{eqnarray}

\begin{defn}
A morphism of BiHom  conformal superalgebras is an even linear $f : A\rightarrow A'$  map which commutes with $\varphi$ and $\psi$, and preserves the structures i.e.
\begin{eqnarray}
\varphi'\circ f=f\circ \varphi,\quad
\psi'\circ f=f\circ \psi,\quad
 f(x\circ_\lambda y)= f(x)\circ'_\lambda f(y), \quad \forall x, y\in A.
\end{eqnarray}
A morphism $f$ of  $A$ into itself is called an endomorphism.
A bijective endomorphism is said to  be an automorphism.
\end{defn}

\begin{defn}
 Given a BiHom  conformal superalgebra $(A, \circ_\lambda, \varphi, \psi)$, we say that
 \begin{enumerate}
  \item [i)] $A$ is multiplicative if the structure maps are endomorphims of $A$.
  \item [ii)] $A$ is regular if the structure maps are automorphisms.
 \end{enumerate}
 \end{defn}

 \begin{defn}
 Given a  BiHom  conformal superalgebra $(A, \circ_\lambda, \varphi, \psi)$, an even $\mathbb{C}[\partial]$-linear map $R : A\rightarrow A$ and $k\in\mathbb{K}$, we say that
$R$ is
a Rota-Baxter operator of weight $k\in\mathbb{K}$ on $A$ if
\begin{eqnarray}
\varphi R&=&R\varphi,\quad \psi R=R\psi,\\
 R(x)\circ_\lambda R(y)&=&R\Big(R(x)\circ_\lambda y+x\circ_\lambda R(y)+k x\circ_\lambda y\Big),
\end{eqnarray}
for all $x, y\in\mathcal{H}(A)$.
 \end{defn}
In this case, the tuple $(A, \circ_\lambda, \varphi, \psi, R)$ is said to be Rota-Baxter BiHom conformal superalgebra.

\begin{defn}
 A BiHom-associative conformal superalgebra is a  BiHom conformal superalgebra $(A, \circ_\lambda, \varphi, \psi)$ such that the following conformal BiHom-associativity holds :
 \begin{eqnarray}
\varphi(x)\circ_\lambda(y\circ_\mu z)=(x\circ_\lambda y)\circ_{\lambda+\mu}\psi(z), \quad \forall x, y, z \in A.
 \end{eqnarray}
\end{defn}

\begin{exam}
 Let $(A, \circ_\lambda)$ be an associative conformal superalgebra and $\varphi, \psi$ two endomorphims of $A$. Then, $(A, \circ_\lambda(\varphi\otimes\psi), \varphi, \psi)$ is a BiHom-associative conformal superalgebra.
\end{exam}

The below important result is proved by direct computation.
\begin{thm}\label{tsa}
 Let $A_{\varphi, \psi}=(A, \circ_\lambda, \varphi, \psi)$ be a BiHom conformal superalgebra and $R$ a Rota-Baxter operator of $A$. Let $\varphi'$ and $\psi'$ be two commuting endomorphims of $A$ that commute with $R$ and $\partial$. Then,
 $$A_{\varphi', \psi'}=(A, \bullet_\lambda=\circ_\lambda(\varphi'\otimes\psi'), \varphi\varphi', \psi\psi')$$
 is a BiHom conformal superalgebra and $R$ is a Rota-Baxter operator on $A_{\varphi', \psi'}$.
 Moreover, if $A_{\varphi, \psi}$ is BiHom-associative, then so is $A_{\varphi', \psi'}$.
\end{thm}
\begin{proof}
It is very easy to prove the sesquilinearity. Let us prove that $R$ is a Rota-Baxter operator on $A_{\varphi', \psi'}$.
We have for any $x, y, z \in \mathcal{H}(A)$,
\begin{eqnarray}
 R(x)\bullet_\lambda R(y)&=&\varphi'(R(x))\circ_\lambda \psi'(R(y))=
 R\varphi'(x)\circ_\lambda R\psi'(y)\nonumber\\
&=& R\Big(R\varphi'(x)\circ_\lambda\psi'(y)+\varphi'(x)\circ_\lambda R\psi'(y)+k\varphi'(x)\circ_\lambda\psi'(y)\Big)\nonumber\\
&=& R\Big(\varphi'R(x)\circ_\lambda\psi'(y)+\varphi'(x)\circ_\lambda \psi'R(y)+k\varphi'(x)\circ_\lambda\psi'(y)\Big)\nonumber\\
&=& R\Big(R(x)\bullet_\lambda y+x\bullet_\lambda R(y)+k x\bullet_\lambda y\Big)\nonumber.
\end{eqnarray}
Now, for the conformal BiHom-associativity, we have for any $x, y, z \in \mathcal{H}(A)$,
\begin{eqnarray}
 \varphi\varphi'(x)\bullet_\lambda(y\bullet_\mu z)
 &=&\varphi'^2\varphi(x)\circ_\lambda\psi'(\varphi'(y)\circ_\mu\psi'(z))
 =\varphi\varphi'^2(x)\circ_\lambda(\varphi'\psi'(y)\circ_\mu\psi'^2(z))\nonumber\\
&=&(\varphi'^2(x)\circ_\lambda\varphi'\psi'(y))\circ_{\lambda+\mu}\psi\psi'^2(z)
 =\varphi'(\varphi'(x)\circ_\lambda\psi'(y))\circ_{\lambda+\mu}\psi'(\psi\psi'(z))\nonumber\\
&=&(x\bullet_\lambda y)\bullet_{\lambda+\mu}\psi\psi'(z)\nonumber.
\end{eqnarray}
This finishes the proof.
\end{proof}

\begin{prop}\label{rac}
Let $A_\circ=(A, \circ_\lambda, \varphi, \psi, R)$ be a Rota-Baxter BiHom-associative conformal superalgebra. Let us define the even bilinear map, for any $x, y\in \mathcal{H}(A)$,
\begin{eqnarray}
 x\bullet_\lambda y=R(x)\circ_\lambda y+x\circ_\lambda R(y)+k x\circ_\lambda y.
\end{eqnarray}
Then, $A_\bullet=(A, \bullet_\lambda, \varphi, \psi, R)$ is also a Rota-Baxter BiHom-associative conformal superalgebra and $R :A_\bullet\rightarrow A_\circ$ is a morphism.
\end{prop}
\begin{proof}
For any $x, y, z \in \mathcal{H}(A)$, we have
\begin{eqnarray}
 &&\hspace{3cm}\varphi(x)\bullet_\lambda(y\bullet_\mu z)=\nonumber\\
 &&=R\varphi(x)\circ_\lambda(y\bullet_\mu z)+\varphi(x)\circ_\lambda R(y\bullet_\mu z)+\varphi(x)\circ_\lambda(y\bullet_\mu z)\nonumber\\
&&= \varphi(x)\circ_\lambda\Big(R(y)\circ_\mu z+y\circ_\mu R(z)+k y\circ_\mu z\Big)+\varphi(x)\circ_\lambda R\Big((R(y)\circ_\mu z+y\circ_\mu R(z)+k y\circ_\mu z)\Big)\nonumber\\
&&\quad+k\varphi(x)\circ_\lambda\Big(R(y)\circ_\mu z+y\circ_\mu R(z)+k y\circ_\mu z\Big)\nonumber\\
&&=(R(x)\circ_\lambda R(y))\circ_{\lambda+\mu}\psi(z)+(R(x)\circ_\lambda y)\circ_{\lambda+\mu}\psi R(z)+k (R(x)\circ_\lambda y)\circ_{\lambda+\mu}\psi(z)\nonumber\\
 &&\qquad+(x\circ_\lambda R(y))\circ_{\lambda+\mu}\psi R(z))+(x\circ_\lambda y)\circ_{\lambda+\mu}\psi R(z)+k (x\circ_\lambda y)\circ_{\lambda+\mu}\psi(z)\nonumber\\
&&\qquad\qquad+k(x\circ_\lambda R(y))\circ_{\lambda+\mu}\psi(z)+k(x\circ_\lambda y)\circ_{\lambda+\mu}\psi R(z)+k^{2}(x\circ_\lambda y)\circ_{\lambda+\mu}\psi(z)\nonumber\\
&&=R(x\bullet_\lambda y)\circ_{\lambda+\mu}\psi(z)+\Big(R(x)\circ_\lambda y+x\circ_\lambda R(y)+k(x\circ_\lambda y)\Big)\circ_{\lambda+\mu}\psi R(z)\nonumber\\
&&\qquad+k\Big(R(x)\circ_\lambda y+ x\circ_\lambda R(y) +x\circ_\lambda y\Big)\circ_{\lambda+\mu}\psi(z)\nonumber\\
&&=R(x\bullet_\lambda y)\circ_{\lambda+\mu}\psi(z)+(x\bullet_\lambda y)\circ_{\lambda+\mu}R\psi(z)+k(x\bullet_\lambda y)\circ_{\lambda+\mu}\psi(z)\nonumber\\
&&=(x\bullet_\lambda y)\bullet_{\lambda+\mu}\psi(z).\nonumber
 \end{eqnarray}
 The second part is a direct consequence of the definition of the new product.
\end{proof}

\begin{defn}
 A BiHom-Lie conformal superalgebra is a BiHom conformal superalgebra $(L, [-_\lambda-], \varphi, \psi)$ such the bracket $[-_\lambda-]$ is skew-symmetric i.e.
  \begin{eqnarray}
{[\psi(x)\circ_\lambda\varphi(y)]=-(-1)^{|x||y|}[\psi(y)_{-\partial-\lambda}\varphi(x)]},\label{sks}
 \end{eqnarray}
 and satisfies the conformal BiHom-Jacobi identity :
  \begin{eqnarray}
(-1)^{|x||z|}[\psi^{2}(x)_\lambda[\psi(y)_\mu\varphi(z)]]
+(-1)^{|x||y|}[\psi^{2}(y)_\mu[\psi(z)_{-\partial-\lambda}\varphi(x)]]\label{bhj}\\
+(-1)^{|y||z|}[\psi^{2}(z)_{-\partial-\lambda-\mu}[\psi(x)_\lambda\varphi(y)]]=0\nonumber,\hspace{2cm}
 \end{eqnarray}
for any homogeneous elements $x, y, z\in \mathcal{H}(A)$,
\end{defn}

A BiHom-Lie conformal superalgebra $(A, [-_\lambda,-], \varphi, \psi)$ is said to be multiplicative if $\varphi$ and $\psi$ are endomorphims of $A$. Whenever $\varphi=\psi$, we recover Hom-Lie conformal superalgebras.

% \begin{exam}
%  arxiv: 2009.03760v1, Example 2.1-2.4.
% \end{exam}

\begin{prop}(\cite{T}, Proposition 2.2)\\
 Let $(A, [-_\lambda-], \varphi, \psi)$ be a BiHom-Lie conformal superalgebra and $\varphi'$ and $\psi'$ be two commuting endomorphims of $A$ which commute also with the structural maps $\varphi$ and $\psi$. Then,
 $$(A, \ldb-_\lambda-\rdb=[-_\lambda-](\varphi'\otimes\psi'), \varphi\varphi', \psi\psi')$$ is also a BiHom-Lie conformal superalgebra.
\end{prop}

\begin{prop}(\cite{T}, Proposition 2.3)\\
 Let $(A, [-_\lambda,-], \varphi, \psi)$ and $(A', [-_\lambda,-]', \varphi', \psi')$ be two BiHom-Lie conformal superalgebras. Then,
 $A\oplus A'$ is a BiHom-Lie conformal superalgebra with :
 \begin{eqnarray}
 (\varphi+\varphi')(x+x')&:=&\varphi(x)+\varphi'(x'),\\
 (\psi+\psi')(x+x')&:=&\psi(x)+\psi'(x'),\\
  \ldb(x+x')_\lambda(y+y')\rdb&:=&[x_\lambda y]+[x'_\lambda y']',
 \end{eqnarray}
for any $x, y\in A, x', y'\in A$.
\end{prop}

\begin{prop}(\cite{T}, Proposition 2.4)\label{alcsa}\\
If $(A, \circ_\lambda, \varphi, \psi)$ is a regular BiHom-associative conformal superalgebra, then
$$(A, [-_\lambda -], \varphi, \psi)$$
is a regular BiHom-Lie conformal superalgebra, where
\begin{eqnarray}
 \ldb x_\lambda y\rdb=x\circ_\lambda y-(-1)^{|x||y|}\varphi^{-1}\psi(y)\circ_{-\lambda-\partial}\varphi\psi^{-1}(x),\nonumber
\end{eqnarray}
for all $x, y\in\mathcal{H}(A)$.
\end{prop}

\begin{prop}\label{rlc}
 Let $(L, [-_\lambda-], \varphi, \psi )$ be a BiHom-Lie conformal superalgebra and $R : L\rightarrow L$ be a Rota-Baxter operator of weight
 $k\in\mathbb{K}$ on $L$. Then,
 $$L_{k, R}=(L, \ldb-_\lambda-\rdb, \varphi, \psi) $$ is again a  BiHom-Lie conformal superalgebra with
\begin{eqnarray}
\ldb x_\lambda y\rdb = [R(x)_\lambda y] + [x_\lambda R(y)] + k[x_\lambda y],
\end{eqnarray}
for all $x, y\in\mathcal{H}(L)$.\\
Moreover, $R$ is a morphism of  $(L, \ldb-_\lambda-\rdb, \varphi, \psi )$
 into $(L, [-_\lambda-], \varphi, \psi )$.
\end{prop}
\begin{proof}
We omit the proof of conformal sesquilinearity. For the conformal skew-symmetry, we have :
\begin{eqnarray}
 \ldb \psi(x)_\lambda \varphi(y)\rdb &=&[R\psi(x)_\lambda \varphi(y)]+[\psi(x)_\lambda R\varphi(y)]+k[\psi(x)_\lambda \varphi(y)]\nonumber\\
 &=&[\psi R(x)_\lambda \varphi(y)]+[\psi(x)_\lambda \varphi R(y)]+k[\psi(x)_\lambda \varphi(y)]\nonumber\\
 &=&-(-1)^{|x||y|}\Big([\psi (y)_{-\lambda-\partial} \varphi R(x)]+[\psi R(y)_{-\lambda-\partial} \varphi (x)]+k[\psi(y)_{-\lambda-\partial} \varphi(x)]\Big)\nonumber\\
  &=&-(-1)^{|x||y|}\Big([\psi (y)_{-\lambda-\partial} R\varphi(x)]+[R\psi (y)_{-\lambda-\partial} \varphi (x)]+k[\psi(y)_{-\lambda-\partial} \varphi(x)]\Big)\nonumber\\
  &=&-(-1)^{|x||y|}\ldb \psi(y)_{-\lambda-\partial} \varphi(x)\rdb.\nonumber
\end{eqnarray}

For the conformal BiHom-Jacobi identity, we have for any $x, y, z\in\mathcal{H}(L)$,
\begin{eqnarray}
 &&\hspace{3cm}\ldb \psi^2(x)_\lambda \ldb\psi(y)_\mu \varphi(z)\rdb\rdb=\nonumber\\
 &&= \ldb \psi^2(x)_\lambda [R\psi(y)_\mu \varphi(z)]+[\psi(y)_\mu R\varphi(z)]+k[\psi(y)_\mu \varphi(z)]\rdb\nonumber\\
 &&= [ R\psi^2(x)_\lambda [R\psi(y)_\mu \varphi(z)]+[\psi(y)_\mu R\varphi(z)]+k[\psi(y)_\mu \varphi(z)]]\nonumber\\
 &&\quad+[\psi^2(x)_\lambda R\Big([R\psi(y)_\mu \varphi(z)]+[\psi(y)_\mu R\varphi(z)]+k[\psi(y)_\mu \varphi(z)]\Big)]\nonumber\\
 &&\qquad+k[\psi^2(x)_\lambda [R\psi(y)_\mu \varphi(z)]
 +[\psi(y)_\mu R\varphi(z)]+k[\psi(y)_\mu \varphi(z)]]\nonumber\\
 &&=[ \psi^2R(x)_\lambda [\psi R(y)_\mu \varphi(z)]]+[ \psi^2R(x)_\lambda [\psi(y)_\mu \varphi R(z)]]+k[ \psi^2R(x)_\lambda [\psi(y)_\mu \varphi(z)]]\nonumber\\
 &&\quad+[\psi^2(x)_\lambda [\psi R(y)_\mu \varphi R(z)]]\nonumber\\
 &&\qquad+k[\psi^2(x)_\lambda [\psi R(y)_\mu \varphi(z)]]
 +k[\psi^2(x)_\lambda [\psi(y)_\mu \varphi R(z)]]+k^2k[\psi^2(x)_\lambda [\psi(y)_\mu \varphi(z)]]\nonumber.
\end{eqnarray}
Similarly,
\begin{eqnarray}
 &&\hspace{3cm}\ldb \psi^2(y)_\mu \ldb\psi(z)_{-\lambda-\partial} \varphi(x)\rdb\rdb=\nonumber\\
 &&=[ \psi^2R(y)_\mu [\psi R(z)_{-\lambda-\partial} \varphi(x)]]+[ \psi^2R(y)_\mu [\psi(z)_{-\lambda-\partial} \varphi R(x)]]+k[ \psi^2R(y)_\mu [\psi(z)_{-\lambda-\partial} \varphi(x)]]\nonumber\\
 &&\quad+[\psi^2(y)_\mu [\psi R(z)_{-\lambda-\partial} \varphi R(x)]]\nonumber\\
 &&\qquad+k[\psi^2(y)_\mu [\psi R(z)_{-\lambda-\partial} \varphi(x)]]
 +k[\psi^2(y)_\mu [\psi(z)_{-\lambda-\partial} \varphi R(x)]]+k^2[\psi^2(y)_\mu [\psi(z)_{-\lambda-\partial} \varphi(x)]].\nonumber
\end{eqnarray}
and
\begin{eqnarray}
 &&\hspace{3cm}\ldb \psi^2(z)_{-\lambda-\mu-\partial} \ldb\psi(x)_{\lambda} \varphi(y)\rdb\rdb=\nonumber\\
 &&=[ \psi^2R(z)_{-\lambda-\mu-\partial} [\psi R(x)_{\lambda} \varphi(y)]]+[ \psi^2R(z)_{-\lambda-\mu-\partial} [\psi(x)_{\lambda} \varphi R(y)]]+k[ \psi^2R(z)_{-\lambda-\mu-\partial} [\psi(x)_{\lambda} \varphi(y)]]\nonumber\\
 &&\quad+[\psi^2(z)_{-\lambda-\mu-\partial} [\psi R(x)_{\lambda} \varphi R(y)]]\nonumber\\
 &&\qquad+k[\psi^2(z)_{-\lambda-\mu-\partial} [\psi R(x)_{\lambda} \varphi(y)]]
 +k[\psi^2(z)_{-\lambda-\mu-\partial} [\psi(x)_{\lambda} \varphi R(y)]]+k^2[\psi^2(z)_{-\lambda-\mu-\partial} [\psi(x)_{\lambda} \varphi(y)]]\nonumber
\end{eqnarray}
Thus, the conformal BiHom-Jacobi identity of the bracket $\ldb-_\lambda-\rdb$ follows from that of $[-_\lambda-]$.
 \end{proof}

% \begin{defn}
%  Left, right BiHom Lie
% \end{defn}

\subsection{BiHom-post-Lie conformal superalgebras}
 \begin{defn}\label{hplad}
 A BiHom-post-Lie conformal superalgebra $(L, [-_\lambda -], \circ_\lambda, \varphi, \psi)$ is a multiplicative BiHom-Lie conformal  superalgebra  $(L, [-_\lambda -], \varphi, \psi)$ and a multiplicative BiHom conformal superalgebra $(L, \circ_\lambda, \varphi, \psi)$ such that
\begin{eqnarray}
 &&\quad\varphi\psi(x)\circ_\lambda [y_\mu z]-[(\psi(x)\circ_\lambda y)_{\lambda+\mu} \psi(z)]-(-1)^{|x||y|}[\psi(y)_\mu (\varphi(x)\circ_\lambda z)]=0, \label{pl4}\\
&&as_{\varphi, \psi}(\psi(x), \varphi(y), z)-(-1)^{|x||y|}as_{\varphi, \psi}(\psi(y), \varphi(x), z)=[\psi(x)_\lambda\varphi(y)]\circ_{\lambda+\mu}\psi(z),\label{pl3}
\end{eqnarray}
where $as_{\varphi, \psi}(x, y, z)=\varphi(x)\circ_\lambda(y\circ_\mu z)-(x\circ_\lambda y)\circ_{\lambda+\mu} \psi(z)$,
for any $x, y, z\in\mathcal{H}(L)$.
\end{defn}

\begin{defn}
A morphism of BiHom-post-Lie  conformal superalgebras is an even $\mathbb{C}$-linear map $f : A\rightarrow A'$  such that, for all $x, y\in A$ :
\begin{eqnarray}
\varphi'\circ f&=&f\circ \varphi,\quad
\psi'\circ f=f\circ \psi,\\
 f([x_\lambda y])&=& [f(x)_\lambda f(y)]',\;
 f(x\circ_\lambda y)= f(x)\circ'_\lambda f(y).
\end{eqnarray}
A morphism $f$ of  $A$ into itself is called an endomorphism.
A bijective endomorphism is said to be an automorphism.
\end{defn}

\begin{exam}
  Let $(L, [-_\lambda -], \circ_\lambda)$ be a post-Lie conformal superalgebra and $\varphi, \psi$ two commuting endomorphims of $L$.
  Then, $(L, \ldb-\lambda-\rdb=[-_\lambda-](\varphi\otimes\psi), \bullet_\lambda(\varphi\otimes\psi), \varphi, \psi)$ is a BiHom-post-Lie conformal superalgebra.
\end{exam}

\begin{prop}
 Let $(L, [-_\lambda -], \circ_\lambda, \varphi, \psi)$ be a BiHom-post-Lie conformal superalgebra. Let us define the $\mathbb{C}$-bilinear maps $\bullet_\lambda : L\times L\rightarrow L[\lambda]$
and $\ldb-_\lambda -\rdb : L\times L\rightarrow L[\lambda]$ as
\begin{eqnarray}
x\bullet_\lambda y:=x\circ_\lambda y+[x_\lambda y],\quad \ldb x_\lambda y\rdb:=-[x_\lambda y].\nonumber
\end{eqnarray}
Then, $(L, \ldb -_\lambda -\rdb, \bullet_\lambda, \varphi, \psi)$ is again a BiHom-post-Lie conformal superalgebra.
\end{prop}

Now, we have the following theorem.
\begin{thm}\label{cla}
 Let $(L, [-_\lambda -], \circ_\lambda, \varphi, \psi)$ be a regular BiHom-post-Lie conformal superalgebra.
Let us define the even $\mathbb{C}$-bilinear map $\ldb-_\lambda -\rdb : L\times L\rightarrow L[\lambda]$ by
\begin{eqnarray}
 \ldb x_\lambda y\rdb:=x\circ_\lambda y-(-1)^{|x||y|} \varphi^{-1}\psi(y)\circ_{-\lambda-\partial}\varphi\psi^{-1}(x)+[x_\lambda y].\nonumber
\end{eqnarray}
Then, $(L, \ldb-_\lambda -\rdb, \varphi, \psi)$ is a BiHom-Lie conformal  superalgebra.
\end{thm}
\begin{proof}
 We have to prove the conformal sesquilinearity, the conformal skew-symmetry and the conformal BiHom-Jacobi identity. But, we only prove the last one.
 \begin{eqnarray}
  \ldb \psi^2(x)_\lambda \ldb \psi(y)_\mu \varphi(z)\rdb\rdb
  &=&\ldb \psi^2(x)_\lambda\Big(\psi(y)\circ_{\mu}\varphi(z)-(-1)^{|y||z|}\psi(z)\circ_{-\mu-\partial}\varphi(y)+[\psi(y)_\mu \varphi(z)]\Big)\rdb\nonumber\\
  &=&\psi^2(x)\circ_\lambda\Big(\psi(y)\circ_{\mu}\varphi(z)-(-1)^{|y||z|}\psi(z)\circ_{-\mu-\partial}\varphi(y)+[\psi(y)_\mu \varphi(z)]\Big)\nonumber\\
  &&-(-1)^{|x|(|y|+|z|)}\varphi^{-1}\psi\Big(\psi(y)\circ_{\mu}\varphi(z)-(-1)^{|y||z|}\psi(z)\circ_{-\mu-\partial}\varphi(y)\nonumber\\
  &&+[\psi(y)_\mu \varphi(z)]\Big)\circ_{-\lambda-\partial}\varphi\psi(x)+[\psi^2(x)_\lambda\Big(\psi(y)\circ_{\mu}\varphi(z)\nonumber\\
  &&-(-1)^{|y||z|}\psi(z)\circ_{-\mu-\partial}\varphi(y)+[\psi(y)_\mu \varphi(z)]\Big)]\nonumber\\
  &=&\psi^2(x)\circ_\lambda(\psi(y)\circ_{\mu}\varphi(z))-(-1)^{|y||z|}\psi^2(x)\circ_\lambda(\psi(z)\circ_{-\mu-\partial}\varphi(y))\nonumber\\
  &&+\psi^2(x)\circ_\lambda([\psi(y)_\mu \varphi(z)])\nonumber\\
  &&-(-1)^{|x|(|y|+|z|)}(\varphi^{-1}\psi^2(y)\circ_{\mu}\psi(z))\circ_{-\lambda-\partial}\varphi\psi(x)\nonumber\\
  &&-(-1)^{|y||z|+|x|(|y|+|z|)}(\varphi^{-1}\psi^2(z)\circ_{-\mu-\partial}\psi(y))\circ_{-\lambda-\partial}\varphi\psi(x)\nonumber\\
  &&+(-1)^{|x|(|y|+|z|)}[\varphi^{-1}\psi^2(y)_\mu \psi(z)]\circ_{-\lambda-\partial}\varphi\psi(x)+[\psi^2(x)_\lambda(\psi(y)\circ_{\mu}\varphi(z))]\nonumber\\
  &&-(-1)^{|y||z|}[\psi^2(x)_\lambda(\psi(z)\circ_{-\mu-\partial}\varphi(y))]+[\psi^2(x)_\lambda[\psi(y)_\mu \varphi(z)]].\nonumber
 \end{eqnarray}
 Next,
 \begin{eqnarray}
  &&\hspace{2cm}\ldb \psi^2(y)_\mu \ldb \psi(z)_{-\lambda-\partial} \varphi(x)\rdb\rdb=\nonumber\\
   &&=\psi^2(y)\circ_\mu(\psi(z)\circ_{-\lambda-\partial}\varphi(x))-(-1)^{|x||y|}\psi^2(y)\circ_\mu(\psi(x)\circ_{\lambda}\varphi(z))\nonumber\\
  &&+\psi^2(y)\circ_\mu([\psi(z)_{-\lambda-\partial} \varphi(x)])-(-1)^{|y|(|x|+|z|)}(\varphi^{-1}\psi^2(z)\circ_{-\lambda-\partial}\psi(x))\circ_{-\mu-\partial}\varphi\psi(y)\nonumber\\
  &&-(-1)^{|x||z|+|y|(|x|+|z|)}(\varphi^{-1}\psi^2(x)\circ_{\lambda}\psi(z))\circ_{-\mu-\partial}\varphi\psi(y)\nonumber\\
  &&+(-1)^{|y|(|x|+|z|)}[\varphi^{-1}\psi^2(z)_{-\lambda-\partial} \psi(x)]\circ_{-\mu-\partial}\varphi\psi(y)\nonumber\\
  &&+[\psi^2(y)_\mu(\psi(z)\circ_{-\lambda-\partial}\varphi(x))]-(-1)^{|x||z|}[\psi^2(y)_\mu(\psi(x)\circ_{\lambda}\varphi(z))]+[\psi^2(y)_\mu[\psi(z)_{-\lambda-\partial} \varphi(x)]],\nonumber
  \end{eqnarray}
  and
  \begin{eqnarray}
 &&\hspace{2cm} \ldb \psi^2(z)_{-\lambda-\mu-\partial} \ldb \psi(x)_\lambda \varphi(y)\rdb\rdb=\nonumber\\
  &&=\psi^2(z)\circ_{-\lambda-\mu-\partial}(\psi(x)\circ_{\lambda}\varphi(y))-(-1)^{|x||y|}\psi^2(z)\circ_{-\lambda-\mu-\partial}(\psi(y)\circ_{-\lambda-\partial}\varphi(x))\nonumber\\
  &&+\psi^2(z)\circ_{-\lambda-\mu-\partial}([\psi(x)_\lambda \varphi(y)])-(-1)^{|z|(|x|+|y|)}(\varphi^{-1}\psi^2(x)\circ_{\lambda}\psi(y))\circ_{\lambda+\mu}\varphi\psi(z)\nonumber\\
  &&-(-1)^{|x||y|+|z|(|x|+|y|)}(\varphi^{-1}\psi^2(y)\circ_{-\lambda-\partial}\psi(x))\circ_{\lambda+\mu}\varphi\psi(z)\nonumber\\
  &&+(-1)^{|z|(|x|+|y|)}[\varphi^{-1}\psi^2(x)_\lambda \psi(y)]\circ_{\lambda+\mu}\varphi\psi(z)+[\psi^2(z)_{-\lambda-\mu-\partial}(\psi(x)\circ_{\lambda}\varphi(y))]\nonumber\\
  &&-(-1)^{|x||y|}[\psi^2(z)_{-\lambda-\mu-\partial}(\psi(y)\circ_{-\lambda-\partial}\varphi(x))]+[\psi^2(z)_{-\lambda-\mu-\partial}[\psi(x)_\lambda \varphi(y)]].\nonumber
  \end{eqnarray}
  The conclusion follows by adding these three equations (with their coefficients) and using axioms of BiHom-post-Lie algebra.
\end{proof}

The proof of the next proposition is similar to the above Theorem \ref{cla}.
 \begin{prop}
 Let $(L, [-_\lambda -], \circ_\lambda, \varphi, \psi)$ be a conformal BiHom-post-Lie superalgebra. With the product
$$x\bullet_\lambda y=x\circ_\lambda y+\frac{1}{2}[x_\lambda y],$$
 $(L, \bullet_\lambda, \varphi, \psi)$ becomes a BiHom-Lie admissible conformal superalgebra i.e. the map $\ast_\lambda : L\times L\rightarrow L[\lambda]$ satisfies the conformal
BiHom-Jacobi identity.
\end{prop}

\begin{thm}\label{tw1p}
 Let $(L, [-_\lambda -], \cdot_\lambda, \varphi, \psi)$ be a  BiHom-post-Lie conformal superalgebra and $\varphi', \psi' : P\rightarrow P$ two even morphisms of
BiHom-post-Lie conformal superalgebras such that the maps $\varphi, \varphi', \psi, \psi'$ commute pairewise. Then,
$$L_{(\varphi', \psi')}=(L, \ldb-_\lambda-\rdb:=[-_\lambda-](\varphi\otimes\psi), \bullet_\lambda:=\circ_\lambda(\varphi\otimes\psi), \varphi\varphi', \psi\psi')$$
is a  BiHom-post-Lie conformal algebra.
\end{thm}
\begin{proof}
 It uses the same thcnics as in the second part of Theorem \ref{tsa}.
\end{proof}

\begin{defn}
An even $\mathbb{C}[\partial]$-linear map $R : L\rightarrow L$ on a BiHom-Lie conformal algebra $(L, [-_\lambda -], \varphi, \psi)$ is said to be a
Rota-Baxter operator of weight $k\in\mathbb{K}$ if $R\varphi=\varphi R$ and $R\psi=\psi R$, and
\begin{eqnarray}
 [R(x)_\lambda R(y)]&=&R\Big([R(x)_\lambda y]+ [x_\lambda R(y)]+k [x_\lambda y]\Big),
\end{eqnarray}
for any $x, y\in\mathcal{H}(T)$.\\
In this case, the tuple $(T, \dashv_\lambda, \vdash_\lambda, \circ_\lambda, \varphi, \psi, R)$ is said to be {\it Rota-Baxter BiHom-tridendriform      conformal superalgebra of weight $k$}.\\
Whenever the product $''\circ_\lambda''$ is trivial, we get BiHom-dendriform superalgebra.
\end{defn}

\begin{prop}
 Let $(L, [-_\lambda -], \varphi, \psi)$ be a regular BiHom-Lie  conformal superalgebra and $R\rightarrow R$ be a Rota-Baxter operator of weight $k$ on $L$. Then, $(L, \ldb-_\lambda -\rdb, \circ_\lambda, \varphi, \psi)$
 is a  BiHom-post-Lie conformal superalgebra, where
$$\ldb x_\lambda y\rdb=k[x_\lambda y],\quad   x\circ_\lambda y=[R(x)_\lambda y],$$ for any $x, y\in\mathcal{H}(L)$.
\end{prop}
\begin{proof}
 For any $x, y, z\in\mathcal{H}(L)$, we have :
\begin{eqnarray}
 &&\varphi\psi(x)\circ_\lambda(\varphi(y)\circ_\mu z)-(\psi(x)\circ_\lambda\varphi(y))\circ_{\lambda+\mu} \psi(z)-(-1)^{|x||y|}\varphi\psi(y) \circ_\lambda(\varphi(x)\circ_\mu z)\nonumber\\
 &&\hspace{2cm}+(-1)^{|x||y|}(\psi(y)\circ_\lambda\varphi(x))\circ_{\lambda+\mu}\psi(z)-
\ldb\psi(x)_\lambda\varphi(y)\rdb\circ_{\lambda+\mu}\psi(z)=\nonumber\\
&&=[R(\varphi\psi(x)_\lambda [R\varphi(y)_\mu z]]-[R([R\psi(x)_\lambda\varphi(y)])_{\lambda+\mu} \psi(z)]-(-1)^{|x||y|}[R(\varphi\psi(y))_\lambda [R\varphi(x)_\mu z]]\nonumber\\
&&\hspace{2cm}+(-1)^{|x||y|}[R([R\psi(y)_\lambda\varphi(x)])_{\lambda+\mu} \psi(z)]-k[R([\psi(x)_\lambda\varphi(y)])_{\lambda+\mu} \psi(z)]\nonumber\\
&&=[[\psi R(x)_\lambda R\varphi(y)]_{\lambda+\mu}\psi(z)]
+(-1)^{|x||y|}[\psi\varphi R(y)_\mu[\varphi R(x)_\lambda z]]-[R([R\psi(x)_\lambda\varphi(y)])_{\lambda+\mu} \psi(z)]\nonumber\\
&&\hspace{1cm}-(-1)^{|x||y|}\Big([[\psi R(y))_\lambda R\varphi(x)]_{\lambda+\mu}\psi(z)]]
+(-1)^{|x||y|}[\psi\varphi R(x)_\mu[\varphi R(y)_\lambda z]] \Big)\nonumber\\
&&\hspace{2cm}+(-1)^{|x||y|}[R([R\psi(y)_\lambda\varphi(x)])_{\lambda+\mu} \psi(z)]-k[R([\psi(x)_\lambda\varphi(y)])_{\lambda+\mu} \psi(z)]\nonumber\\
&&=[[\psi R(x)_\lambda R\varphi(y)]_{\lambda+\mu}\psi(z)]
-[R([R\psi(x)_\lambda\varphi(y)])_{\lambda+\mu} \psi(z)]\nonumber\\
&&\hspace{2cm}+(-1)^{|x||y|}[R([R\psi(y)_\lambda\varphi(x)])_{\lambda+\mu} \psi(z)]-k[R([\psi(x)_\lambda\varphi(y)])_{\lambda+\mu} \psi(z)]\nonumber\\
&&=[\Big([\psi R(x)_\lambda R\varphi(y)]
-R([R\psi(x)_\lambda\varphi(y)]-[R\psi(x)_\lambda\varphi(y)]-k[\psi(x)_\lambda\varphi(y)]\Big)_{\lambda+\mu} \psi(z)].\nonumber
\end{eqnarray}
The last line vanishes by Rota-Baxter identity. Now,
 \begin{eqnarray}
&&\varphi\psi(x)\circ_\lambda[y_\mu z]-[(\psi(x)\circ_\lambda y)_{\lambda+\mu} \psi(z)]- (-1)^{|x||y|}[\psi(y)_\mu(\psi(x)\circ_\lambda z)]=\nonumber\\
&&= [R\varphi\psi(x)_\lambda[y_\mu z]]-[[R\psi(x)_\lambda  y]_{\lambda+\mu} \psi(z)]- (-1)^{|x||y|}[\psi(y)_\mu [R\psi(x)_\lambda z]]\nonumber.
 \end{eqnarray}
The conformal BiHom-Jacobi identity of the bracket $[-_\lambda-]$ ends the proof.
\end{proof}

\begin{thm}
 Let $(L, [-_\lambda -], \circ_\lambda, \varphi, \psi)$ be a BiHom-post-Lie conformal superalgebra. Let us define the new bracket
 \begin{eqnarray}
  \ldb(a, x)_\lambda (b, y)\rdb&=&\Big(a\circ_\lambda b-(-1)^{|a||b|} \varphi^{-1}\psi(b)\circ_{-\lambda-\partial}\varphi\psi^{-1}(a)+[a_\lambda b],\nonumber\\
  &&a\circ_\lambda y-(-1)^{|x||y|} \varphi^{-1}\psi(b)\circ_{-\lambda-\partial}\varphi\psi^{-1}(x)+[x_\lambda y]\Big),\nonumber
 \end{eqnarray}
for all $a, b, x, y\in \mathcal{H}(L)$. Then, $(L, \ldb-_\lambda-\rdb, \varphi^{\times 2}, \psi^{\times 2})$ is a BiHom-Lie conformal superalgebra.
\end{thm}
\begin{proof}
 It comes from direct computation by using axioms in Definition \ref{hplad}.
\end{proof}

\subsection{BiHom-tridendriform conformal superalgebras}
\begin{defn}\label{tdd}
 A  BiHom-tridendriform conformal superalgebra is a tuple $(T, \dashv, \vdash, \circ, \varphi, \psi)$ consisting of a
 $\mathbb{C}[\partial]$-vector superspace $T$,
 three even $\mathbb{C}$-bilinear maps  $\dashv_\lambda, \vdash_\lambda, \circ_\lambda : T\otimes T\rightarrow T[\lambda]$,
 and  two even $\mathbb{C}$-linear map $\varphi, \psi : T\rightarrow T$ satisfying
\begin{eqnarray}
\varphi\psi&=&\psi\varphi,
\\
\varphi(x\dashv_\lambda y)&=&\varphi(x)\dashv_\lambda\varphi(y), \quad
\psi(x\dashv_\lambda y)=\psi(x)\dashv_\lambda\psi(y),
\\
\varphi(x\vdash_\lambda y)&=&\varphi(x)\vdash_\lambda\varphi(y),\quad
\psi(x\vdash_\lambda y)=\psi(x)\vdash_\lambda\psi(y),
\\
\varphi(x\circ_\lambda y)&=&\varphi(x)\circ_\lambda\varphi(y),
 \quad \psi(x\circ_\lambda y)=\psi(x)\circ_\lambda\psi(y), \\
(x\dashv_\lambda y)\dashv_{\lambda+\mu}\psi(z)&=&\varphi(x)\dashv_\lambda(y\dashv_\mu z+(-1)^{|y||z|}y\vdash_\mu z+(-1)^{|y||z|}y\circ_\mu z),\label{fat}\\
(x\vdash_\lambda y)\dashv_{\lambda+\mu}\psi(z)&=&(-1)^{|x||z|}\varphi(x)\vdash_\lambda(y\dashv_\mu z),\\
\varphi(x)\vdash_\lambda(y\vdash_\mu z)&=&((-1)^{|x||y|}x\dashv_\lambda y+x\vdash_\lambda y+x\circ_\lambda y)\vdash_{\lambda+\mu}\psi(z),\\
(x\dashv_\lambda y)\circ_{\lambda+\mu}\psi(z)&=&(-1)^{|x||y|}\varphi(x)\circ_\lambda(y\vdash_\mu z),\\
(x\vdash_\lambda y)\circ_{\lambda+\mu}\psi(z)&=&\varphi(x)\vdash_\lambda(y\circ_\mu z),\\
(x\circ_\lambda y)\dashv_{\lambda+\mu}\psi(z)&=&(-1)^{|x||z|}\varphi(x)\circ_\lambda(y\dashv_\mu z),\\
(x\circ_\lambda y)\circ_{\lambda+\mu}\psi(z)&=&\varphi(x)\circ_\lambda(y\circ_\mu z),
\end{eqnarray}
for $x, y, z\in \mathcal{H}(T)$.
 \end{defn}

The proof of the next proposition is also similar to the one of Theorem \ref{tsa}.
\begin{prop}
 Let $(T, \dashv_\lambda, \vdash_\lambda, \circ_\lambda, \varphi, \psi)$ be a  BiHom-tridendriform conformal superalgebra. Let $\varphi'$ and $\psi'$ be two commuting endomorphims of $T$ that commute with $\partial$. Then,
 $$T_{\varphi', \psi'}=(T, \triangleleft_\lambda=\dashv_\lambda(\varphi'\otimes\psi'), \triangleright_\lambda=\vdash_\lambda(\varphi'\otimes\psi'),
 \bullet_\lambda=\circ_\lambda(\varphi'\otimes\psi'), \varphi\varphi', \psi\psi')$$
 is a BiHom-tridendriform conformal superalgebra.
\end{prop}

\begin{defn}
A morphism of BiHom-tridendriform  conformal superalgebras is an even $\mathbb{C}$-linear map $f : T\rightarrow T'$  such that, for all $x, y\in T$ :
\begin{eqnarray}
\varphi'\circ f&=&f\circ \varphi,\quad
\psi'\circ f=f\circ \psi,\\
 f(x\dashv_\lambda y)&=& f(x)\dashv'_\lambda f(y),\;
 f(x\vdash_\lambda y)= f(x)\vdash'_\lambda f(y),\;f(x\circ_\lambda y)= f(x)\circ'_\lambda f(y).
\end{eqnarray}
A morphism $f$ of  $T$ into itself is called an endomorphim.
A bijective endomorphism is said to be an automorphism.
\end{defn}

\begin{defn}
An even $\mathbb{C}[\partial]$-linear map $R : T\rightarrow T$ on a BiHom-tridendriform conformal superalgebra $(T, \dashv_\lambda, \vdash_\lambda, \circ_\lambda, \varphi, \psi)$ is said to be a
Rota-Baxter operator of weight $k\in\mathbb{K}$ if $R$ commutes with $\varphi$ and $\psi$, and
\begin{eqnarray}
 R(x)\dashv_\lambda R(y)&=&R\Big(R(x)\dashv_\lambda y+ x\dashv_\lambda R(y)+k x\dashv_\lambda y\Big),\\
 R(x)\vdash_\lambda R(y)&=&R\Big(R(x)\vdash_\lambda y+x\vdash_\lambda R(y)+k x\vdash_\lambda y\Big),\\
 R(x)\circ_\lambda R(y)&=&R\Big(R(x)\circ_\lambda y+x\circ_\lambda R(y)+k x\circ_\lambda y\Big),\label{initl}
\end{eqnarray}
for any $x, y\in\mathcal{H}(T)$.\\
In this case, the tuple $(T, \dashv_\lambda, \vdash_\lambda, \circ_\lambda, \varphi, \psi, R)$ is said to be {\it Rota-Baxter BiHom-tridendriform      conformal superalgebra of weight $k$}.\\
Whenever the product $''\circ_\lambda$ is trivial, we say get BiHom-dendriform superalgebra.
\end{defn}

\begin{thm}\label{car1}
 Let $(A, \circ_\lambda, \varphi, \psi, R)$ be a Rota-Baxter BiHom-associative conformal superalgebra of weight $k$.
Let us define three new operations $\dashv_\lambda, \vdash_\lambda$ and
$\ast_\lambda$ on $A$ by
$$x\dashv_\lambda y:=x\circ_\lambda R(y), \quad x\vdash_\lambda y:=(-1)^{|x||y|}R(x)\circ_\lambda y\quad\mbox{and}\quad x\ast_\lambda y:=(-1)^{|x||y|}k x\circ y,$$

Then, $(A, \dashv_\lambda, \vdash_\lambda, \ast_\lambda, \varphi, \psi)$ is a BiHom-tridendriform conformal superalgebra.
\end{thm}
\begin{proof}
Let us prove axiom (\ref{fat}). For any $x, y\in\mathcal{H}(A)$, we have
\begin{eqnarray}
(x\dashv_\lambda y)\dashv_{\lambda+\mu}\psi(z)&=&(x\circ_\lambda R(y))\circ_{\lambda+\mu}R\psi(z)\nonumber\\
&=&\varphi(x)\circ_\lambda (R(y)\circ_{\mu}R(z))\nonumber\\
&=&\varphi(x)\circ_\lambda R\Big(R(y)\circ_\lambda z+y\circ_\lambda R(z)+k y\circ_\lambda z\Big)\nonumber\\
&=&\varphi(x)\dashv_\lambda((-1)^{|y||z|}y\vdash_\mu z+y\dashv_\mu z+(-1)^{|y||z|}y\ast_\mu z).\nonumber
\end{eqnarray}
The others are proved in the same way.
\end{proof}

\begin{prop}
 Let $(A, \circ_\lambda, \varphi, \psi, R)$ be a Rota-Baxter BiHom-associative conformal superalgebra of weight $k$. Let us
define the multiplications $\dashv_\lambda$ and $\vdash_\lambda$ by
$$x\dashv_\lambda y:=x\circ_\lambda R(y)+k(-1)^{|x||y|} x\circ_\lambda y \quad\mbox{and}\quad x\vdash_\lambda y:=(-1)^{|x||y|}R(x)\circ_\lambda y.$$
Then $(A, \dashv_\lambda, \vdash_\lambda, \varphi, \psi)$ is a BiHom-dendriform conformal superalgebra.
\end{prop}
% % % % \begin{proof}
% % % % The axioms are checked as in the proof of Proposition \ref{car1}.
% % % % \end{proof}
\begin{cor}
 Let $(A, \circ_\lambda, \varphi, \psi, R)$ be a Rota-Baxter BiHom-associative conformal superalgebra of weight $0$. We define new even operations as :
$\dashv_\lambda : A\times A\rightarrow A$ and $\vdash_\lambda : A\times A\rightarrow A$ on $A$ by
$$x\dashv_\lambda y:=x\circ_\lambda R(y)\quad\mbox{and}\quad x\vdash_\lambda y:=(-1)^{|x||y|}R(x)\circ_\lambda y.$$
Then, $(A, \dashv_\lambda, \vdash_\lambda, \varphi, \psi)$ is a BiHom-dendriform conformal superalgebra.
\end{cor}

\begin{cor}
 Let $(A, \circ_\lambda, \varphi, \psi, R)$ be a Rota-Baxter BiHom-associative conformal superalgebra of weight $k$. We define new even operations as :
$\dashv_\lambda : A\times A\rightarrow A$ and $\vdash_\lambda : A\times A\rightarrow A$ on $A$ by
$$x\dashv_\lambda y:=x\circ_\lambda R(y)+k x\circ_\lambda y\quad\mbox{and}\quad x\vdash_\lambda y:=(-1)^{|x||y|}R(x)\circ_\lambda y.$$
Then, $(A, \dashv_\lambda, \vdash_\lambda, \varphi, \psi)$ is a BiHom-dendriform conformal superalgebra.
\end{cor}

 The following result connects BiHom-tridendriform conformal superalgebras to BiHom-post-Lie conformal superalgebras.
\begin{prop}\label{tp}
Let $(T, \dashv_\lambda, \vdash_\lambda, \circ_\lambda, \varphi, \psi)$ be a regular  BiHom-tridendriform conformal superalgebra. Then, $(T, \ldb-_\lambda-\rdb, \bullet_\lambda, \varphi, \psi)$ is
a BiHom-post-Lie conformal superalgebra with
\begin{eqnarray}
 \ldb x_\lambda y\rdb&=&x\circ_\lambda y-(-1)^{|x||y|}\varphi^{-1}\psi(y)\circ_{-\lambda-\partial}\varphi\psi^{-1}(x),\\
 x\bullet_\lambda y&=&x\vdash_\lambda y-(-1)^{|x||y|}\varphi^{-1}\psi(y)\dashv_{-\lambda-\partial}\varphi\psi^{-1}(x),
\end{eqnarray}
 for any $x, y\in\mathcal{H}(T)$.
\end{prop}
\begin{proof}
It is clear that $(T, \ldb-_\lambda-\rdb, \varphi, \psi)$ is a BiHom-Lie conformal superalgebra by Proposition \ref{alcsa}.
 Let us prove the first axiom. For any any $x, y, z\in\mathcal{H}(T)$, we have
 \begin{eqnarray}
 &&\hspace{1cm}\varphi\psi(x)\bullet_\lambda\ldb y_\mu z\rdb- \ldb (\psi(x)\bullet_\lambda y)_{\lambda+\mu} \psi(z)\rdb -(-)^{|x||y|}\ldb \psi(y)_\mu(\varphi(x)\bullet_\lambda z)\rdb  =\nonumber\\
 &&=\varphi\psi(x)\vdash_\lambda \ldb y_\mu z\rdb -(-1)^{|x|(|y|+|z|)} \varphi^{-1}\psi(\ldb y_\mu z\rdb)\dashv_{-\lambda-\partial}\varphi^2(x)\nonumber\\
&&\quad -(\psi(x)\bullet_\lambda y)\circ_{\lambda+\mu}\psi(z)+(-1)^{(|x|+|y|)|z|}\varphi^{-1}\psi^2(z)\circ_{-\lambda-\mu-\partial} \varphi\psi^{-1}(\psi(x)\bullet_\lambda y)\nonumber\\
&&\qquad-\psi(y)\circ_\mu (\varphi(x)\bullet_\lambda z)+(-1)^{|y|(|x|+|z|)}\varphi^{-1}\psi(\varphi(x)\bullet_\lambda z)_{-\mu-\partial}\varphi\psi^{-1}\psi(y)\nonumber\\
&&=\varphi\psi(x)\vdash_\lambda(y\circ_\mu z)-(-1)^{|y||z|}\varphi\psi(x)\dashv_{-\lambda-\partial}(\varphi^{-1}\psi(z)\circ_{-\mu-\partial} \varphi\psi^{-1}(y))\nonumber\\
&&-(-1)^{|x|(|y|+|z|)}(\varphi^{-1}\psi(y)\circ_\mu \varphi^{-1}\psi(z))\dashv_{-\lambda-\partial}\varphi^2+
(-1)^{|x||y|+|y||z|+|x||z|}(\varphi^{-2}\psi^2\circ_{-\mu-\partial} y)\dashv_{-\lambda-\mu}\varphi^2(x)\nonumber\\
&&\quad-(\psi(x)\vdash_\lambda y)\circ_{\lambda+\mu}\psi(z)+(-1)^{|x||y|}(\varphi^{-1}\psi(y)\dashv_{-\lambda-\partial} \varphi(x))\circ_{\lambda+\mu}\psi(z)\nonumber\\
&&+(-1)^{(|x|+|y|)|z|}\varphi^{-1}\psi^2(z)\circ_{-\lambda-\mu-\partial}(\varphi(x)\vdash_\lambda \varphi\psi^{-1}(y))\nonumber\\
&&-(-1)^{|x||y|+|y||z|+|x||z|}\varphi^{-1}\psi^2(z)\circ_{-\lambda-\mu-\partial}(y\dashv_{-\lambda-\partial}\varphi^2\psi^{-1}(x))\nonumber\\
&&\quad-(-)^{|x||y|}\psi(y)\circ_\mu(\varphi(x)\dashv_\lambda z)+(-1)^{|x|(|y|+|z|)}\psi(y)\circ_\mu(\varphi^{-1}\psi(z)\dashv_{-\lambda-\partial}\varphi^2\psi^{-1}(x))\nonumber\\
&&+(-1)^{|y||z|}(\psi(x)\dashv_\lambda \varphi^{-1}\psi(z))\circ_{-\mu-\partial}\psi(y)-(-1)^{|x||z|+|y||z|}(\varphi^{-1}\psi(z)\dashv_{-\lambda-\partial}\psi(x))\circ_{-\mu-\partial}\psi(y)\nonumber.
 \end{eqnarray}
Which cancels by axioms in Definition \ref{tdd}. The last axiom is proved similarly.
\end{proof}

\begin{cor}(\cite{T}, Proposition 3.12)\\
Let $(T, \dashv, \vdash, \circ, \varphi, \psi)$ be a regular  BiHom-tridendriform algebra. Then, $(T, \ldb-, -\rdb, \bullet, \varphi, \psi)$ is
a BiHom-post-Lie  algebra with
\begin{eqnarray}
 \ldb x, y\rdb&=&x\circ y-(-1)^{|x||y|}\varphi^{-1}\psi(y)\circ\varphi\psi^{-1}(x),\\
 x\bullet y&=&x\vdash y-(-1)^{|x||y|}\varphi^{-1}\psi(y)\dashv\varphi\psi^{-1}(x),
\end{eqnarray}
 for any $x, y\in\mathcal{H}(T)$.
\end{cor}

\begin{thm}\label{tp}
Let $(T, \dashv_\lambda, \vdash_\lambda, \circ_\lambda, \varphi, \psi)$ be a  BiHom-tridendriform conformal superalgebra. Then, $(T, \bullet_\lambda, \varphi, \psi)$ is
a BiHom-associative conformal superalgebra, where
 $$x\bullet_\lambda y=x\vdash_\lambda y+(-1)^{|x||y|}x\dashv_\lambda y+ x\circ_\lambda y, $$
 for all $x, y, z\in\mathcal{H}(T)$.
\end{thm}
\begin{proof}
The sesquilinearity of the new product is easy to prove. Now,
 for any $x, y, z\in\mathcal{H}(T)$,
 \begin{eqnarray}
 &&\hspace{2cm}\varphi(x)_\lambda\bullet(y\bullet_\mu z)-(x\bullet_\lambda y)\bullet_{\lambda+\mu} \psi(z)=\nonumber\\
 &&=\varphi(x)\vdash_\lambda(y\vdash_\mu z)+(-1)^{|y||z|}\varphi(x)\vdash_\lambda(y\dashv_\mu z)+\varphi(x)\vdash_\lambda(y\circ_\mu z)\nonumber\\
&&+(-1)^{|x||y|}(-1)^{|x||z|}\varphi(x)\dashv_\lambda(y\vdash z)
+(-1)^{|x||y|+|x||z|+|y||z|}\varphi(x)\dashv_\lambda(y\dashv_\mu z)\nonumber\\
&&\quad+(-1)^{|x|(|y|+|z|)}\varphi(x)\dashv_\lambda(y\circ_\mu z)+\varphi(x)\circ_\lambda(y\vdash_\mu z)+(-1)^{|y||z|}\varphi(x)\circ_\lambda(y\dashv_\mu z)\nonumber\\
&&
+\varphi(x)\circ_\lambda(y\circ_\mu z)-(x\vdash_\lambda y)\vdash_{\lambda+\mu}\psi(z)-(-1)^{|x||y|}(x\dashv_\lambda y)\vdash_{\lambda+\mu}\psi(z)-(x\circ_\lambda y)\vdash_{\lambda+\mu}\psi(z)\nonumber\\
&&\quad-(-1)^{(|x|+|y|)|z|}(x\vdash_\lambda y)\dashv_{\lambda+\mu}\psi(z)-(-1)^{|x||z|+|y||z|+|x||y|}(x\dashv_\lambda y)\dashv_{\lambda+\mu}\psi(z)\nonumber\\
&&-(-1)^{(|x|+|y|)|z|}(x\circ_\lambda y)\dashv_{\lambda+\mu}\psi(z)-(x\vdash_\lambda y)\circ_{\lambda+\mu}\psi(z)-(-1)^{|x||y|}(x\dashv_\lambda y)\circ_{\lambda+\mu}\psi(z)\nonumber\\
&&\quad-(x\circ_\lambda y)\circ_{\lambda+\mu}\psi(z).\nonumber
 \end{eqnarray}
By axioms in Definition \ref{tdd}, the proof follows.
\end{proof}

\begin{cor}
Let $(T, \dashv_\lambda, \vdash_\lambda, \circ_\lambda, \varphi, \psi)$ be a  BiHom-tridendriform conformal algebra. Then, $(T, \ldb-_\lambda-\rdb, \bullet_\lambda,  \varphi, \psi)$ is
a BiHom-Poisson conformal algebra, where
 \begin{eqnarray}
 x\bullet_\lambda y&=&x\vdash_\lambda y+(-1)^{|x||y|}x\dashv_\lambda y+ x\circ_\lambda y,\\
 \ldb x_\lambda y\rdb&=&x\bullet_\lambda y-(-1)^{|x||y|}\varphi^{-1}\psi(y)\bullet_{-\lambda-\partial}\varphi\psi^{-1}(x),
 \end{eqnarray}
 for all $x, y, z\in\mathcal{H}(T)$.
\end{cor}
\begin{proof}
 It follows from Proposition \ref{alcsa} and Theorem \ref{tp}.
\end{proof}

Now we have the following definition.
\begin{defn}
 A commutative  BiHom-tridendriform conformal superalgebra is a tuple $(T, \star, \ast, \varphi, \psi)$ in which $(T, \ast,  \varphi, \psi)$
 is a commutative BiHom-associative conformal superalgebra and
  $\star : T\otimes T\rightarrow T$ is an even bilinear map such that :
\begin{eqnarray}
\varphi(x)\star_\lambda(y\ast_\mu z)&=&(x\star_\lambda y)\ast_{\lambda+\mu}\psi(z),\qquad\qquad\quad\label{d2}\\
 \varphi(x)\star_\lambda(y\star_\mu z)&=&(x\star_\lambda y+ (-1)^{|x||y|} y\star_{-\lambda-\partial} x+x\ast_\lambda y)\star_{\lambda+\mu}\psi(z),\label{d1}
\end{eqnarray}
for any $x, y, z\in\mathcal{H}(T)$.
\end{defn}

\begin{prop}\label{cca1}
Let $(A, \star_\lambda, \ast_\lambda, \varphi)$ be a commutative  Hom-tridendriform conformal superalgebra.
Let us define the even bilinear map $\bullet_\lambda : A\times A\rightarrow A[\lambda]$ by
\begin{eqnarray}
 x\bullet_\lambda y:=x\star_\lambda y+(-1)^{|x||y|} y\star_{-\lambda-\partial} x+x\ast_\lambda y.\nonumber
\end{eqnarray}
Then, $(A, \bullet_\lambda, \varphi)$ is a commutative Hom-associative conformal superalgebra.
\end{prop}
\begin{proof}
It is clear that the product $\bullet_\lambda$ is commutative. Now, we prove the conformal Hom-associativity.
One has for any $x, y, z\in\mathcal{H}(A)$,
 \begin{eqnarray}
 &&\hspace{2cm} (x\bullet_\lambda y)\bullet_{\lambda+\mu}\varphi(z)-\varphi(x)\bullet_\lambda( y\bullet_\mu z)=\nonumber\\
&&=(x\star_\lambda y+(-1)^{|x||y|} y\star_{-\lambda-\partial} x+x\ast_\lambda y)\star_{\lambda+\mu}\varphi(z)\nonumber\\
&&+(-1)^{(|x|+|y|)|z|}\varphi(z)\star_{-\lambda-\mu-\partial} (x\star_\lambda y+(-1)^{|x||y|}y\star_{-\lambda-\partial} x+x\ast_\lambda y)\nonumber\\
&&+(x\star_\lambda y+(-1)^{|x||y|}y\star_{-\lambda-\partial} x+x\ast_\lambda y)\ast_{\lambda+\mu}\varphi(z)\nonumber\\
&&-\varphi(x)\star_\lambda(y\star_\mu z+(-1)^{|y||z|}z\star_{-\mu-\partial} y+y\ast_\mu z)\nonumber\\
&&-(-1)^{|x|(|y|+|z|)}(y\star_\mu z+(-1)^{|y||z|}z\star_{-\mu-\partial} y+y\ast_\mu z)\star_{-\lambda-\partial}\varphi(x)\nonumber\\
&&-\varphi(x)\ast_\lambda(y\star_\mu z+(-1)^{|y||z|}z\star_{-\mu-\partial} y+y\ast_\mu z).\nonumber
\end{eqnarray}
By expanding the rigth hand side and grouping the terms, we get
\begin{eqnarray}
&&\hspace{3cm}(x\bullet_\lambda y)\bullet_{\lambda+\mu}\varphi(z)-\varphi(x)\bullet_\lambda( y\bullet_\mu z)=\nonumber\\
&&=\Big((x\star_\lambda y)\star_{\lambda+\mu}\varphi(z)+(-1)^{|x||y|} (y\star_{-\lambda-\partial} x)\star_{\lambda+\mu}\varphi(z)+(x\ast_\lambda y)\star_{\lambda+\mu}\varphi(z)-\varphi(x)\star_\lambda(y\star_\mu z)\Big)\nonumber\\
&&+\Big((-1)^{|x||y|+(|x|+|y|)|z|}\varphi(z)\star_{-\lambda-\mu-\partial}(y\star_{-\lambda-\partial} x)-(-1)^{|x|(|y|+|z|)}(y\star_\mu z)\star_{-\lambda-\partial}\varphi(x)\nonumber\\
&&-(-1)^{|x|(|y|+|z|)+|y||z|}(z\star_{-\mu-\partial} y)\star_{-\lambda-\partial}\varphi(x)
-(-1)^{|x|(|y|+|z|)}(y\ast_\mu z)\star_{-\lambda-\partial}\varphi(x)\Big)\nonumber\\
&&+\Big((-1)^{(|x|+|y|)|z|}\varphi(z)\star_{-\lambda-\mu-\partial}(x\ast_\lambda y)-(-1)^{|y||z|}\varphi(x)\ast_\lambda(z\star_{-\mu-\partial} y)\Big)+\Big((x\star_\lambda y)\ast_{\lambda+\mu}\varphi(z)\nonumber\\
&&-\varphi(x)\star_\lambda(y\ast_\mu z)\Big)+\Big((x\ast_\lambda y)\ast_{\lambda+\mu}\varphi(z)-\varphi(x)\ast_\lambda(y\ast_\mu z)\Big)
\nonumber\\
&&+(-1)^{(|x|+|y|)|z|}\varphi(z)\star_{-\lambda-\mu-\partial}(x\star_\lambda y)-\varphi(x)\ast_\lambda(y\star_\mu z)
+(-1)^{|x||y|}(y\star_{-\lambda-\partial} x)\ast_{\lambda+\mu}\varphi(z)\nonumber\\
&&-(-1)^{|y||z|}\varphi(x)\star_\lambda(z\star_{-\mu-\partial} y)
\nonumber.
\end{eqnarray}
By commutativity and (\ref{d1}) the first two brackets cancels. By commutativity and (\ref{d2}) the next two ones cancel. The next one cancels  by conformal Hom-associativity. Thus,
\begin{eqnarray}
(x\bullet_\lambda y)\bullet_{\lambda+\mu}\varphi(z)-\varphi(x)\bullet_\lambda( y\bullet_\mu z)&=&
+(-1)^{(|x|+|y|)|z|}\varphi(z)\star_{-\lambda-\mu-\partial}(x\star_\lambda y)-\varphi(x)\ast_\lambda(y\star_\mu z)\nonumber\\
&&+(-1)^{|x||y|}(y\star_{-\lambda-\partial} x)\ast_{\lambda+\mu}\varphi(z)\nonumber\\
&&\quad-(-1)^{|y||z|}\varphi(x)\star_\lambda(z\star_{-\mu-\partial} y)\nonumber.
\end{eqnarray}
Now, by adding and subtracting pairewise below four terms, and reorganizing, we obtain :
\begin{eqnarray}
&&\hspace{3cm}(x\bullet_\lambda y)\bullet_{\lambda+\mu}\varphi(z)-\varphi(x)\bullet_\lambda( y\bullet_\mu z)=\nonumber\\
&&=\Big((-1)^{(|x|+|y|)|z|}\varphi(z)\star_{-\lambda-\mu-\partial}(x\star_\lambda y)-(-1)^{|y||z|}(x\star_{-\lambda-\mu-\partial} z)\star_{-\mu-\partial}\varphi(y)\nonumber\\
&&\qquad-(-1)^{(|x|+|y|)|z|}(z\ast_{\lambda+\mu} x)\star_{-\mu-\partial}\varphi(y)-(-1)^{(|x|+|y|)|z|}(z\star_{-\lambda-\mu-\partial} x)\star_{-\mu-\partial}\varphi(y)\Big)\nonumber\\
&&\quad+\Big(-(-1)^{|y||z|}\varphi(x)\star_\lambda(z\star_{-\mu-\partial} y)+(-1)^{|y||z|}(x\star_{-\lambda-\mu-\partial} z)\star_{-\mu-\partial}\varphi(y)\nonumber\\
&&\qquad+(-1)^{(|x|+|y|)|z|}(z\ast_{\lambda+\mu} x)\star_{-\mu-\partial}\varphi(y)+(-1)^{(|x|+|y|)|z|}(z\star_{-\lambda-\mu-\partial} x)\star_{-\mu-\partial}\varphi(y)\Big)\nonumber\\
&&\quad+\Big((-1)^{|x||y|}(y\star_{-\lambda-\partial} x)\ast_{\lambda+\mu}\varphi(z)+(-1)^{|x||y|}\varphi(y)\star_\mu(x\ast_{\lambda+\mu} z)\Big)\nonumber\\
&&\qquad+\Big(-\varphi(x)\ast_\lambda(y\star_\mu z)-(-1)^{|x||y|}\varphi(y)\star_\mu(x\ast_{\lambda+\mu} z)\Big)
\nonumber.
 \end{eqnarray}
By commutativity and (\ref{d1}) the first two brackets cancel and by commutativity and (\ref{d2}) the last two ones cancel also.
\end{proof}

\section{Construction theorems}
\begin{defn}
 A BiHom-Poisson conformal superalgebra is quintuple $(P, \ast_\lambda, [-_\lambda -], \varphi, \psi)$, where
 \begin{enumerate}
  \item [i)]
 $(P, \ast_\lambda, \varphi, \psi)$ is a BiHom-associative conformal superalgebra
 \item [ii)]
 $(P, [-_\lambda -], \varphi, \psi)$ is a BiHom-Lie conformal superalgebra, and
 \item [iii)]
 the conformal BiHom-Leibniz identity holds  i.e.
 \begin{eqnarray}
 {[\varphi\psi(x)_\lambda(y\ast_\mu z)]=[\psi(x)_\lambda y]\ast_{\lambda+\mu} \psi(z)+(-1)^{|x||y|}\psi(y)\ast_\mu[\varphi(x)_\lambda z]},
 \end{eqnarray}
 \end{enumerate}
 for any $x, y\in\mathcal{H}(P)$.
\end{defn}

\begin{exam}
Let $(P, \ast, [-, -], \varphi, \psi)$ be an ordinary  BiHom-Poisson superalgebra. Then,
$P=\mathbb{C}[\partial]\otimes A$, with $\partial\varphi=\varphi\partial$ and $\partial\psi=\psi\partial$, equipped with $\lambda$-operations
$$x\ast_\lambda y=x\ast y,\quad [x_\lambda y]=[x, y], \forall x, y \in \mathcal{H}(P),$$
is a BiHom-Poisson conformal superalgebra.
\end{exam}

\begin{exam}
 Let $(P', \circ'_\lambda, [-_\lambda -]', Id)$ be a  BiHom-Poisson conformal superalgebra and $P$ a vector superspace together
 with  a $\mathbb{C}[\partial]$-linear bijective map $f : P\rightarrow P'$  such that
$$f(x\circ_\lambda y)=f(x)\circ' f(y)\quad\mbox{and}\quad f([x_\lambda y])=[f(x)_\lambda f(y)]',\; \forall x, y\in \mathcal{H}(P).$$
Then, $(P, \circ_\lambda, [-_\lambda -],  Id)$ is a  BiHom-Poisson conformal superalgebra.
\end{exam}

\begin{defn}
Let $(P, \ast_\lambda, [-_\lambda -], \varphi, \psi)$ be a  BiHom-Poisson conformal superalgebra.  An even $\mathbb{C}[\partial]$-linear map $R : P\rightarrow P$  is said to be a Rota-Baxter operator on $P$ if it is
a Rota-Baxter operator for each of the products
$\ast_\lambda$ and $[-_\lambda -]$.
\end{defn}

Now, in what follows, we give some constructions of BiHom-Poisson conformal superalgebras from other structures.
\subsection{From BiHom-associative conformal superalgebras}

\begin{thm}
If $(A, \circ_\lambda, \varphi, \psi)$ is a regular BiHom-associative conformal superalgebra, then
$$(A, \circ_\lambda, [-_\lambda -], \varphi, \psi)$$
is a regular BiHom-Poisson conformal superalgebra, where
\begin{eqnarray}
 [x_\lambda y]=x\circ_\lambda y-(-1)^{|x||y|}\varphi^{-1}\psi(y)\circ_{-\lambda-\partial}\varphi\psi^{-1}(x),\nonumber
\end{eqnarray}
for all $x, y, z\in\mathcal{H}(A)$.
\end{thm}
\begin{proof}
 As $(A, \circ_\lambda, \varphi, \psi)$ is a regular BiHom-associative conformal superalgebra and $(L, [-_\lambda-], \varphi, \psi)$ is a regular BiHom-Lie conformal superalgebra \cite{T}, we only have to prove the conformal BiHom-Leibniz identity. For any $x, y, z\in\mathcal{H}(A)$,
 \begin{eqnarray}
  [\varphi\psi(x)_\lambda(y\circ_\mu z)]&=&
  \varphi\psi(x)\circ_\lambda(y\circ_\mu z)-(-1)^{|x|(|y|+|z|)}(\varphi^{-1}\psi(y)\circ_{\mu}\varphi^{-1}\psi(z))\circ_{-\lambda-\partial}\varphi^2(x),\nonumber\\
  {[\psi(x)_\lambda y]}\circ_{\lambda+\mu}\psi(z)&=&(\psi(x)\circ_\lambda y)\circ_{\lambda+\mu}\psi(z)-(-1)^{|x||y|}(\varphi^{-1}\psi(y)\circ_{-\lambda-\partial}\varphi(x))\circ_{\lambda+\mu}\psi(z),\nonumber\\
  \psi(y)\circ_\mu[\varphi(x)_\lambda z]&=&\psi(y)\circ_\mu(\varphi(x)\circ_\lambda z)-(-1)^{|x||z|}\psi(y)\circ_\mu(\varphi^{-1}\psi(z)\circ_{-\lambda-\partial}\varphi^2\psi^{-1}(x)).\nonumber
 \end{eqnarray}
By regular conformal BiHom-associativity the conclusion follows.
\end{proof}

\begin{cor}
 If $(A, \circ_\lambda, \varphi, \psi)$ is a regular BiHom-associative conformal algebra, then
$$(A, \circ_\lambda, [-_\lambda -], \varphi, \psi)$$
is a regular BiHom-Poisson conformal algebra, where
\begin{eqnarray}
 [x_\lambda y]=x\circ_\lambda y-\varphi^{-1}\psi(y)\circ_{-\lambda-\partial}\varphi\psi^{-1}(x),\nonumber
\end{eqnarray}
for all $x, y, z\in A$.
\end{cor}

\begin{cor}
 If $(A, \circ, \varphi, \psi)$ is a regular BiHom-associative  algebra, then
$$(A, \circ, [-, -], \varphi, \psi)$$
is a regular BiHom-Poisson algebra, where
\begin{eqnarray}
 [x, y]=x\circ y-\varphi^{-1}\psi(y)\circ\varphi\psi^{-1}(x),\nonumber
\end{eqnarray}
for all $x, y, z\in A$.
\end{cor}

\begin{cor}
 If $(A, \circ, \varphi, \psi)$ is a regular BiHom-associative superalgebra, then
$$(A, \circ, [-, -], \varphi, \psi)$$
is a regular BiHom-Poisson  superalgebra, where
\begin{eqnarray}
 [x, y]=x\circ y-(-1)^{|x||y|}\varphi^{-1}\psi(y)\circ\varphi\psi^{-1}(x),\nonumber
\end{eqnarray}
for all $x, y, z\in\mathcal{H}(A)$.
\end{cor}

\subsection{From Hom-post-Poisson conformal superalgebras}

Here, we consider the Hom-type case.
\begin{defn}\label{hppad}
 A Hom-post-Poisson conformal superalgebra is a $\mathbb{Z}_2$-graded vector space $P$ equipped with four even $\mathbb{C}-$bilinear maps
$[-_\lambda -] : P\otimes P\rightarrow P[\lambda]$, $\circ_\lambda : P\otimes P\rightarrow P[\lambda]$, $\star_\lambda : P\otimes P\rightarrow P[\lambda]$ and $\ast_\lambda : P\otimes P\rightarrow P[\lambda]$,
 two an $\mathbb{C}$-linear maps $\varphi : P\rightarrow P$  such that
\begin{enumerate}
 \item [1)]
$(P, [-_\lambda -], \circ_\lambda, \varphi)$
 is a Hom-post-Lie conformal superalgebra,
\item [2)] $(P, \star_\lambda, \ast_\lambda, \varphi)$
is a commutative  Hom-tridendriform conformal superalgebra, and they are compatible in the sense that
\item [3)]
\begin{eqnarray}
 &&[\varphi(x)_\lambda y\ast_\lambda z]=[x_\lambda y]\ast_\lambda\varphi(z)+(-1)^{|x||y|}\varphi(y)\ast_\lambda[x_\lambda z],\label{p1}\\
 &&\varphi(x)\star_\lambda[y_\lambda z]=(x\circ_\lambda y)\ast_\lambda\varphi(z)+(-1)^{|x||y|}[\varphi(y)_\lambda x\star_\lambda z],\label{p2}\\
&&\varphi(x)\circ_\lambda(y\ast_\mu z)=(x\circ_\lambda y)\ast_{\lambda+\mu}\varphi(z)+(-1)^{|x||y|}\varphi(y)\ast_\lambda(x\circ_\lambda z),\label{p3}\\
&&\varphi(x)\star_\lambda(y\circ_\mu z)\!\!+\!\!(-1)^{|x||y|}\varphi(y)\star_\lambda(x\circ_\mu z)=\Big(x\star_\lambda y\!\!+\!\!(-1)^{|x||y|}y\star_\lambda x+x\ast_\lambda y\Big)\circ_{\lambda+\mu}\varphi(z)
,\label{p4}\\
 &&\quad\varphi(x)\circ_\lambda (y\star_\mu z)-(-1)^{|x||y|}\varphi(y)\star_\lambda(x\circ_\mu z)=\label{pp5}\\ &&\hspace{3cm}=\Big(x\circ_\lambda y-(-1)^{|x||y|}y\circ_\lambda x+[x_\lambda y]\Big)\star_{\lambda+\mu}\varphi(z),\nonumber
\end{eqnarray}
\end{enumerate}
for any $x, y, z\in\mathcal{H}(P)$.
\end{defn}

Whenever the product ''$\cdot$'' and the bracket ''$[-, -]$ are trivial i.e. $x\cdot y=[x, y]=0$ for all $x, y\in P$, we say that $(P, ...)$ is a BiHom-prepoisson algebra.

% \begin{exam}
%  arxiv:2209.12242v1 (page 11).
% \end{exam}

\begin{thm}
 Let $(P, [-_\lambda -], \circ_\lambda, \star_\lambda, \ast_\lambda, \varphi, \psi)$ be a  BiHom-post-Poisson conformal superalgebra. Let us define two new  even bilinear operations
\begin{eqnarray}
x\bullet_\lambda y&=&x\star_\lambda y+(-1)^{|x||y|}y\star_{-\lambda-\partial} x+x\ast_\lambda y, \nonumber\\
\ldb x_\lambda y\rdb&=&x\circ_\lambda y-(-1)^{|x||y|}y\circ_{-\lambda-\partial}x+[x_\lambda y],\nonumber
\end{eqnarray}
for any $x, y\in \mathcal{H}(P)$. Then, $(P, \bullet_\lambda, \ldb-_\lambda-\rdb, \varphi)$ is a commutative Hom-Poisson conformal  superalgebra.
\end{thm}
{\it Proof :}
The commutative conformal associativity comes from Proposition \ref{cca1}  and the Hom-Lie conformal structure follows from  Lemma \ref{cla}.
Let us prove the Hom-Leibniz conformal rule. For any $x, y, z\in\mathcal{H}({P})$,
\begin{eqnarray}
 &&\quad\ldb\varphi(x)_\lambda (y\bullet_{\mu} z)\rdb-\ldb x_\lambda y\rdb\bullet_{\lambda+\mu}\varphi(z)-(-1)^{|x||y|}\varphi(y)\bullet_\mu\ldb x_\lambda z\rdb=\nonumber\\
&&=\varphi(x)\circ_\lambda( y\star_\mu z+(-1)^{|y||z|}z\star_{-\mu-\partial} y+y\ast_\mu z)\nonumber\\
&&\quad-(-1)^{|x|(|y|+|z|)}( y\star_\mu z+(-1)^{|y||z|}z\star_{-\mu-\partial} y+y\ast_\mu z)\circ_{-\lambda-\partial}\varphi(x)\nonumber\\
&&\quad+[\varphi(x)_\lambda(y\star _\mu z+(-1)^{|y||z|}z\star_{-\mu-\partial} y+y\ast_\mu z)]\nonumber\\
&&-(x\circ_\lambda y-(-1)^{|x||y|}y\circ_{-\lambda-\partial} x+[x_\lambda y])\star_{\lambda+\mu}\varphi(z)\nonumber\\
&&-(-1)^{(|x|+|y|)|z|}\varphi(z)\star_{-\lambda-\mu-\partial}(x\circ_\lambda y-(-1)^{|x||y|}y\circ_{-\lambda-\partial} x+[x_\lambda y])\nonumber\\
&&-(x\circ_\lambda y-(-1)^{|x||y|}y\circ_{-\lambda-\partial} x+[x_\lambda y])\ast_{\lambda+\mu}\varphi(z)\nonumber\\
&&\quad-(-1)^{|x||y|}\Big(\varphi(y)\star_\mu(x\circ_\lambda z-(-1)^{|x||z|}z\circ_{-\lambda-\partial} x+[x_\lambda z])\nonumber\\
&&\quad+(-1)^{|y|(|x|+|z|)}(x\circ_\lambda z-(-1)^{|x||z|}z\circ_{-\lambda-\partial} x+[x_\lambda z])\star_{-\mu-\partial}\varphi(y)\nonumber\\
&&\quad+\varphi(y)\ast_\mu(x\circ_\lambda z-(-1)^{|x||z|}z\circ_{-\lambda-\partial} x+[x_\lambda z])\Big)\nonumber
\end{eqnarray}
By expending and grouping the terms, we get :
\begin{eqnarray}
 &&\quad\ldb\varphi(x)_\lambda (y\bullet_{\mu} z)\rdb-\ldb x_\lambda y\rdb\bullet_{\lambda+\mu}\varphi(z)-(-1)^{|x||y|}\varphi(y)\bullet_\mu\ldb x_\lambda z\rdb=\nonumber\\
&&=\Big(\varphi(x)\circ_\lambda(y\star_\mu z)-(x\circ_\lambda y)\star_{\lambda+\mu}\varphi(z)+(-1)^{|x||y|}(y\circ_{-\lambda-\partial} x)\star_{\lambda+\mu}\varphi(z)
-[x_\lambda y]\star_{\lambda+\mu}\varphi(z)\nonumber\\
&&-(-1)^{|x||y|}\varphi(y)\star_\mu(x\circ_\lambda z)\Big)+\Big((-1)^{|y||z|}\varphi(x)\circ_\lambda(z\star_{-\mu-\partial} y)-(-1)^{(|x|+|y|)|z|}\varphi(z)\star_{-\lambda-\mu-\partial}(x\circ_\lambda y)\nonumber\\
&&-(-1)^{|y||z|}(x\circ_\lambda z)\star_{-\mu-\partial}\varphi(y)
+(-1)^{(|x|+|y|)|z|}(z\circ_{-\lambda-\partial} x)\star_{-\lambda-\partial}\varphi(y)-(-1)^{|y||z|}[x_\lambda z]\star_{-\mu-\partial}\varphi(y)\Big)\nonumber\\
&&+\Big(\varphi(x)\circ_\lambda(y\ast_\mu z)-(x\circ_\lambda y)\ast_{\lambda+\mu}\varphi(z)-(-1)^{|x||y|}\varphi(y)\ast_\mu(x\circ_\lambda z)\Big)\nonumber\\
&&+\Big(-(-1)^{|x|(|y|+|z|)}(y\star_\mu z)\circ_{-\lambda-\partial}\varphi(x)-(-1)^{|x|(|y|+|z|)+|y||z|}(z\star_{-\mu-\partial} y)\circ_{-\lambda-\partial}\varphi(x)\nonumber\\
&&-(-1)^{|x|(|y|+|z|)}(y\ast_\mu z)\circ_{-\lambda-\partial}\varphi(x)\nonumber\\
&&+(-1)^{(|x|+|y|)|z|+|x||y|}\varphi(z)\star_{-\lambda-\mu-\partial}(y\circ_{-\lambda-\partial} x)+(-1)^{|x|(|y|+|z|)}\varphi(y)\star_\mu(z\circ_{-\lambda-\partial} x)\Big)
\nonumber\\
&&+\Big([\varphi(x)_\lambda (y\star_\mu z)]+(-1)^{|x||y|}(y\circ_{-\lambda-\partial} x)\ast_{\lambda+\mu}\varphi(z)-(-1)^{|x||y|}\varphi(y)\star_\mu[x_\lambda z]\Big)\nonumber\\
&&+\Big((-1)^{|y||z|}[\varphi(x)_\lambda (z\star_{-\mu-\partial} y)]-(-1)^{(|x|+|y|)|z|}\varphi(z)\star_{-\lambda-\mu-\partial}[x_\lambda y]\nonumber\\
&&+(-1)^{|x|(|y|+|z|)}\varphi(y)\ast_\mu(z\circ_{-\lambda-\mu} x)\Big)+\Big([\varphi(x)_\lambda (y\ast_\mu z)]-[x_\lambda y]\ast_{\lambda+\mu}\varphi(z)
\nonumber\\
&&-(-1)^{|x||y|}\varphi(y)_\mu\ast[x_\lambda z]\Big).\nonumber
\end{eqnarray}
The first two brackets cancel by (\ref{pp5}), the third by (\ref{p3}), the next by
 (\ref{p4}), the fith and sixth by (\ref{p2}) and commutativity, the last by (\ref{p1}).

We can thus conclude that $(P, \bullet_\lambda, \ldb-_\lambda-\rdb, \varphi, \psi)$ is a commutative Hom-Poisson conformal  superalgebra.
\begin{cor}
If $(P, \star_\lambda, \circ_\lambda, \varphi, \psi)$ be a Hom-pre-Poisson conformal superalgebra, then\\ $(P, \ast_\lambda, \ldb-_\lambda-\rdb, \varphi, \psi)$ is a commutative
Hom-Poisson conformal  superalgebra.
\end{cor}

\subsection{From  BiHom-Poisson conformal superalgebras and twisting}
The following theorem asserts that  BiHom-Poisson color algebra turn to another one via morphisms.
\begin{thm}\label{tw1}
 Let $(P, \circ_\lambda, [-_\lambda-], \varphi, \psi )$ be a  BiHom-Poisson conformal superalgebra and $\varphi', \psi' : P\rightarrow P$ two even morphisms of
BiHom-Poisson conformal superalgebras such that the maps $\varphi, \varphi', \psi, \psi'$ commute pairewise. Then,
$$P_{(\varphi', \psi')}=(P, \bullet_\lambda:=\circ_\lambda(\varphi\otimes\psi), \ldb-_\lambda-\rdb:=[-_\lambda-](\varphi\otimes\psi), \varphi\varphi', \psi\psi')$$
is a  BiHom-Poisson conformal superalgebra.

Moreover, if $\tilde P=(\bar P, \tilde\circ_\lambda, \tilde{[-_\lambda-]}, \tilde\varphi, \tilde\psi)$ and $P_{(\tilde\varphi', \tilde\psi')}$ is another  BiHom-Poisson conformal superalgebra and $f : P\rightarrow \tilde P$ is a morphism such that $f$ commutes with $\tilde\varphi'$ and $\tilde\psi'$, Then, $f$ is also a morphism from $P_{(\varphi', \psi')}$ to
$P_{(\tilde\varphi', \tilde\psi')}$.
\end{thm}
\begin{proof}
 It has the same technics as in Theorem \ref{tsa}.
\end{proof}

\begin{cor}\label{c1tw1}
 Let $(P, \circ_\lambda, [-_\lambda-], \varphi, \psi )$ be a  BiHom-Poisson conformal superalgebra. Then,
 $$P_{(\varphi^{n}, \psi^{n})}=(P, \bullet_\lambda:=\circ_\lambda(\varphi\otimes\psi), \ldb-_\lambda-\rdb:=[-_\lambda-](\varphi\otimes\psi), \varphi^{n+1}, \psi^{n+1})$$
is also a  BiHom-Poisson conformal superalgebra.
\end{cor}
\begin{proof}
 It suffises to take $\varphi'=\varphi^n$ and $\psi'=\psi^n$ in Theorem \ref{tw1}.
\end{proof}

By taking $\varphi=\varphi'\neq id\neq\psi=\psi'$ in Theorem \ref{tw1} or taking $n=1$ in Corollary \ref{c1tw1} we get the following corollary.

 \begin{cor}
 Let $(P, \circ_\lambda, [-_\lambda-], \varphi, \psi )$ be a BiHom-Poisson conformal superalgebra. Then,
 $$P_{(\varphi, \psi)}=(P, \bullet_\lambda:=\circ_\lambda(\varphi\otimes\psi), \ldb-_\lambda-\rdb:=[-_\lambda-](\varphi\otimes\psi), \varphi^{2}, \psi^{2})$$
is also a BiHom-Poisson conformal superalgebra.
\end{cor}

By taking $\varphi=\psi\neq id\neq\varphi'=\psi'$ in Theorem \ref{tw1}, we get the following corollary.

\begin{cor}
 Let $(P, \circ_\lambda, [-_\lambda-], \varphi, \psi )$ be a BiHom-Poisson conformal superalgebra. Then,
 $$P_{(\varphi, \psi)}=(P, \bullet_\lambda:=\circ_\lambda(\varphi\otimes\varphi), \ldb-_\lambda-\rdb:=[-_\lambda-](\varphi\otimes\varphi), \varphi\psi)$$
is a Hom-Poisson conformal superalgebra.
\end{cor}

\begin{cor}
If we suppose that the BiHom-Poisson conformal superalgebra is regular and the morphisms are bijective in  Theorem \ref{tw1}, and taking
$\varphi'=\varphi^{-1} \neq id\neq\psi'=\psi^{-1}$ get that
$$P_{(\varphi^{-1}, \psi^{-1})}=(P, \bullet_\lambda:=\circ_\lambda(\varphi\otimes\psi), \ldb-_\lambda-\rdb:=[-_\lambda-](\varphi\otimes\psi))$$
is a Poisson conformal superalgebra.
\end{cor}

By taking $\varphi=\psi=id$ and then putting $\varphi'=\varphi$ and $\psi'=\psi$ in Theorem \ref{tw1}, we obtain :
 \begin{cor}
 Let $(P, \circ_\lambda, [-_\lambda-])$ be a Poisson conformal superalgebra. Then,
$$P_{(\varphi, \psi)}=(P, \bullet_\lambda:=\circ_\lambda, \ldb-_\lambda-\rdb:=[-_\lambda-], \varphi, \psi)$$
is a  BiHom-Poisson conformal superalgebra.
\end{cor}

By taking $\varphi'=\psi^n=\varphi^n$ in Theorem \ref{tw1}, we get :

\begin{cor}\label{tw2}
 Let $(P, \circ_\lambda, [-_\lambda-], \varphi, \psi )$ be a multiplicative BiHom-Poisson conformal superalgebra. Then, for any on negative integer $n$,
$$P_n=(P, \bullet^{n}_\lambda=\varphi^{n}\circ_\lambda, \ldb-_\lambda-\rdb^{n}=\psi^{n}[-_\lambda-], \varphi^{n+1}, \varphi^{n}\psi)$$
is a  BiHom-Poisson conformal superalgebra.
\end{cor}

Whenever $G=\{e\}$, we recover
\begin{cor}(\cite{S}, Theorem 1)\\
 Let $(P, \circ_\lambda, [-_\lambda-], \varphi, \psi )$ be a  BiHom-Poisson conformal algebra and $\varphi', \psi' : P\rightarrow P$ two morphisms of
BiHom-Poisson conformal algebras such that the maps $\varphi, \varphi', \psi, \psi'$ commute pairewise. Then,
$$P_{(\varphi', \psi')}=(P, \bullet_\lambda:=\circ_\lambda(\varphi\otimes\psi), \ldb-_\lambda-\rdb:=[-_\lambda-](\varphi\otimes\psi), \varphi\varphi', \psi\psi')$$
is a  BiHom-Poisson conformal algebra.

Moreover, if $\tilde P=(\bar P, \tilde\circ_\lambda, \tilde{[-_\lambda-]}, \tilde\varphi, \tilde\psi)$ and $P_{(\tilde\varphi', \tilde\psi')}$ is another  BiHom-Poisson conformal algebra and $f : P\rightarrow \tilde P$ is a morphism such that $f$ commutes with $\tilde\varphi'$ and $\tilde\psi'$, Then, $f$ is also a morphism from $P_{(\varphi', \psi')}$ to
$P_{(\tilde\varphi', \tilde\psi')}$.
\end{cor}

Now, we give a result of twisting by averaging operator.

\begin{thm}\label{bk1}
Let $(P,\circ, [-, -], \varphi, \psi)$ be a BiHom-Poisson conformal superalgebra and $\beta : P\rightarrow P$ an averaging operator i.e. an even $\mathbb{C}[\partial]$-linear map that commutes with $\varphi$ and $\psi$, and
\begin{eqnarray}
 \beta(\beta(x)\circ_\lambda y)=\beta(x)\circ_\lambda\beta(y)=\beta(x\circ_\lambda \beta(y))\quad\mbox{and}\quad \beta([\beta(x)_\lambda y])=[\beta(x)_\lambda\beta(y)]=\beta([x_\lambda \beta(y)]).\nonumber
\end{eqnarray}

Let us define two new operations $\bullet_\lambda : P\times P\rightarrow P[\lambda]$ and $\ldb-_\lambda, -\rdb : P\times P\rightarrow P[\lambda]$ by
$$x\bullet_\lambda y:= \beta(x)\circ_\lambda y, \quad\ldb x_\lambda y\rdb:=[\beta(x)_\lambda y].$$
Then, $(P, \bullet_\lambda, \ldb-_\lambda- \rdb, \varphi, \psi)$ is also a BiHom-Poisson conformal superalgebra.
\end{thm}
\begin{proof}
 The proof is straighforward by a simple computation.
\end{proof}

We end this subsection by twisting by Rota-Baxter operator.

\begin{thm}\label{R}
 Let $(P, \circ_\lambda, [-_\lambda-], \varphi, \psi )$ be a BiHom-Poisson conformal superalgebra and $R : P\rightarrow P$ be a Rota-Baxter operator of weight
 $k\in\mathbb{K}$ on $P$. Then,
 $$P_{k, R}=(P, \bullet_\lambda, \ldb-_\lambda-\rdb, \varphi, \psi) $$ is again a  BiHom-Poisson conformal superalgebra with
\begin{eqnarray}
x\bullet_\lambda y &=& R(x)\circ_\lambda y + x\circ_\lambda R(y) +k x\circ_\lambda y,\\
\ldb x_\lambda y\rdb &=& [R(x)_\lambda y] + [x_\lambda R(y)] + k[x_\lambda y],
\end{eqnarray}
for all $x, y\in\mathcal{H}(P)$.\\
Moreover, $R$ is a morphism of  $(P, \bullet_\lambda, \ldb-_\lambda-\rdb, \varphi, \psi )$
 into $(P, \circ_\lambda, [-_\lambda-], \varphi, \psi )$.
\end{thm}
\begin{proof}
The BiHom-associativity follows from Theorem \ref{rac} and the BiHom-Lie structure from Proposition \ref{rlc}. Now, we prove the Leibniz conformal identity. For any $x, y, z\in \mathcal{H}(P)$,
\begin{eqnarray}
 \ldb\psi\varphi(x)_{\lambda}(y\bullet_{\mu} z)\rdb
&=&\ldb\psi\varphi(x)_{\lambda}\Big(R(y)\circ_{\mu}z+y\circ_{\mu}R(z)+k y\circ_{\mu} z\Big)\rdb\nonumber\\
&=&[R\psi(x)_{\lambda}\Big(R(y)\circ_{\mu}z+y\circ_{\mu}R(z)+k y\circ_{\mu}z\Big)]\nonumber\\
&&+[\psi(x)_{\lambda}R\Big(R(y)\circ_{\mu}z+y\circ_{\mu}R(z)+k y\circ_{\mu}z\Big)]\nonumber\\
&&+k[\psi(x)_{\lambda}\Big(R(y)\circ_{\mu}z+y\circ_{\mu}R(z)+k y\circ_{\mu}z\Big)]\nonumber\\
&=&[R\psi(x)_{\lambda} (R(y)\circ_{\mu}z)]+[R\psi(x)_{\lambda} (y\circ_{\mu}R(z))]+k[R\psi(x)_{\lambda} (y\circ_{\mu}z)]\nonumber\\
&&+[\psi(x)_{\lambda} (R(y)\circ_{\mu}R(z))]\nonumber\\
&&+k[\psi(x)_{\lambda} (R(y)\circ_{\mu}z)]+k[\psi(x)_{\lambda} (y\circ_{\mu}R(z))]+k^2[\psi(x)_{\lambda} (y\circ_{\mu}z)]\nonumber.
\end{eqnarray}
By conformal BiHom-Leibniz identity,
\begin{eqnarray}
 \ldb\psi\varphi(x)_{\lambda} (y\bullet_{\mu} z)\rdb
&=&[R\psi(x)_{\lambda} R(y)]_{\lambda+\mu}\psi(z)+(-1)^{|x||y|}\psi(R(y))\circ_{\mu}[R\varphi(x)_{\lambda} z]\nonumber\\ &&+[R\psi(x)_{\lambda} y]\circ_{\lambda+\mu}\psi(R(z))+(-1)^{|x||y|}\psi(y)\circ_{\mu}[R\varphi(x)_{\lambda} R(z)]\nonumber\\
&&+k[R\psi(x)_{\lambda} y]\circ_{\lambda+\mu}\psi(z)+\lambda(-1)^{|x||y|}\psi(y)\circ_\mu[R\varphi(x)_{\lambda} z]\nonumber\\
&&+[\psi(x)_{\lambda} R(y)]\circ_{\lambda+\mu}\psi(R(z))+(-1)^{|x||y|}\psi(R(y))\circ_{\mu}[\varphi(x)_{\lambda} R(z)]\nonumber\\
&&+k[\psi(x)_{\lambda} R(y)]\circ_{\lambda+\mu}\psi(z)+(-1)^{|x||y|}k\psi R(y)\circ_{\mu}[\varphi(x)_{\lambda} z]\nonumber\\
&&+k[\psi(x)_{\lambda} y]\circ_{\lambda+\mu}\psi(R(z))+(-1)^{|x||y|}k\psi(y)\circ_{\mu}[\varphi(x)_{\lambda} R(z)]\nonumber\\
&&+k^2[\psi(x)_{\lambda} y]\circ_{\lambda+\mu}\psi(z)+
(-1)^{|x||y|}k^2\psi(y)\circ_{\mu}[\varphi(x)_{\lambda} z]\nonumber.
\end{eqnarray}
Furthermore,

\begin{eqnarray}
&&\ldb\psi(x)_{\lambda} y\rdb\bullet_{\lambda+\mu}\psi(z)+(-1)^{|x||y|}\psi(y)\bullet_{\mu}\ldb\varphi(x)_{\lambda} z\rdb=\nonumber\\
&&=\Big([R\psi(x)_{\lambda} y]+[\psi(x)_{\lambda} R(y)]+k[\psi(x)_{\lambda} y]\Big)\bullet_{\lambda+\mu}\psi(z)\nonumber\\
&&\quad+(-1)^{|x||y|}\psi(y)\bullet_{\mu}\Big([R\varphi(x)_{\lambda} z]+[\varphi(x)_{\lambda} R(z)]+k[\varphi(x)_{\lambda} z]\Big).\nonumber\\
&&=R\Big([R\psi(x)_{\lambda} y]+[\psi(x)_{\lambda} R(y)]+k[\psi(x)_{\lambda} y]\Big)\circ_{\lambda+\mu}\psi(z)\nonumber\\
&&\quad+\Big([R\psi(x)_{\lambda} y]+[\psi(x)_{\lambda} R(y)]+k[\psi(x)_{\lambda} y]\Big)\circ_{\lambda+\mu}R\psi(z)\nonumber\\
&&\quad+k\Big([R\psi(x)_{\lambda} y]+[\psi(x)_{\lambda} R(y)]+k[\psi(x)_{\lambda} y]\Big)\circ_{\lambda+\mu}\psi(z)\nonumber\\
&&\quad+(-1)^{|x||y|}R\psi(y)\circ_{\mu}\Big([R\varphi(x)_{\lambda} z]+[\varphi(x)_{\lambda} R(z)]+k[\varphi(x)_{\lambda} z]\Big)\nonumber\\
&&+(-1)^{|x||y|}\psi(y)\circ_{\mu}R\Big([R\varphi(x)_{\lambda} z]+[\varphi(x)_{\lambda} R(z)]+k[\varphi(x)_{\lambda} z])\nonumber\\
&&\quad+(-1)^{|x||y|}k\psi(y)\circ_{\mu}\Big([R\varphi(x)_{\lambda} z]+[\varphi(x)_{\lambda} R(z)]+k[\varphi(x)_{\lambda} z]\Big).\nonumber
\end{eqnarray}
\begin{eqnarray}
&&\ldb\psi(x)_{\lambda} y\rdb\bullet_{\lambda+\mu}\psi(z)+(-1)^{|x||y|}\psi(y)\bullet_{\mu}\ldb\varphi(x)_{\lambda} z\rdb=\nonumber\\
&&={[R\psi(x)_{\lambda} R(y)]}\circ_{\lambda+\mu}\psi(z)\nonumber\\
&&\quad+[R\psi(x)_{\lambda} y]\circ_{\lambda+\mu}\psi(R(z))+[\psi(x)_{\lambda} R(y)]\circ_{\lambda+\mu}\psi(R(z))+k[\psi(x)_{\lambda} y]\circ_{\lambda+\mu}\psi(R(z))\nonumber\\
&&\quad+k[R\psi(x)_{\lambda} y]\circ_{\lambda+\mu}\psi(z)+k[\psi(x)_{\lambda} R(y)]\circ_{\lambda+\mu}\psi(z)+k^2[\psi(x)_{\lambda} y]\circ_{\lambda+\mu}\psi(z)\nonumber\\
&&\quad+(-1)^{|x||y|}\psi(R(y))\circ_{\mu}[R\varphi(x)_{\lambda} z]+(-1)^{|x||y|}\psi(R(y))\circ_{\mu}[\varphi(x)_{\lambda} R(z)]\nonumber\\
&&\quad+(-1)^{|x||y|}k\psi R(y)\circ_{\mu}[\varphi(x)_{\lambda} z]+(-1)^{|x||y|}\psi(y)\circ_{\mu}[R\varphi(x)_{\lambda} R(z)]\nonumber\\
&&\quad+(-1)^{|x||y|}k\psi(y)\circ_{\mu}[R\varphi(x)_{\lambda} z]
+(-1)^{|x||y|}k\psi(y)\circ_{\mu}[\varphi(x)_{\lambda} R(z)]\nonumber\\
&&\quad+(-1)^{|x||y|}k^2\psi(y)\circ_{\mu}[\varphi(x)_{\lambda} z]\nonumber.
\end{eqnarray}
A simple comparaison allows us to draw the conclusion.
\end{proof}
Whenever $\lambda=0$, and $G=\{e\}$, we recover
\begin{cor}(\cite{S}, Proposition 5)\\
Under the Assumption of  Theorem \ref{R},
$$P_{0, R}=(P, \bullet_\lambda, \ldb-_\lambda-\rdb, \varphi, \psi) $$ is a  BiHom-Poisson conformal algebra.
\end{cor}

With $\psi=\varphi$ in Theorem\ref{R}, we get
\begin{cor}
 If $(P, \circ_\lambda, [-_\lambda-], \varphi)$ is a Hom-Poisson conformal superalgebra and $R : P\rightarrow P$ is a Rota-Baxter operator of weight
 $k\in\mathbb{K}$ on $P$. Then,
 $$P_{k, R}=(P, \bullet_\lambda, \ldb-_\lambda-\rdb, \varphi) $$ is also a Hom-Poisson conformal superalgebra.
\end{cor}

\subsection{From extensions}

\begin{thm}
 Let $(P', \circ'_\lambda, [-_\lambda,-]', \varphi', \psi')$ and $(P'', \circ''_\lambda, [-_\lambda,-]'', \varphi'', \psi''´)$ be two BiHom-Poisson conformal superalgebras. Then,
 $$P=(P'\oplus P'', \circ_\lambda, [ -_\lambda -], \varphi, \psi)$$ is a BiHom-Poisson conformal superalgebra with :
 \begin{eqnarray}
 \partial &=&\partial'\oplus\partial'',\\
 \varphi=(\varphi''+\varphi'')(x''+x'')&=&\varphi'(x')+\varphi''(x''),\\
 \psi=(\psi'+\psi'')(x'+x'')&=&\psi'(x')+\psi''(x''),\\
 {(x'+x'')\circ_\lambda(y'+y'')}&:=& x'\circ'_\lambda y' +
 y'\circ''_\lambda y'',\\
  {}[(x'+x'')_\lambda(y'+y'')]&:=&[x'_\lambda y']'+[x''_\lambda y'']'',
 \end{eqnarray}
for any $x', y'\in P', x'', y''\in P''$.

% Moreover, if is a linear map. Then,
% is a morphism if and only if its graph
% is a BiHom-subalgebra of [see On BiHom-associative algebras].
\end{thm}

Whenever $G=\{e\}$, we recover
\begin{cor}(\cite{S}, Lemma 1)\\
 Let $(P', \circ'_\lambda, [-_\lambda,-]', \varphi', \psi')$ and $(P'', \circ''_\lambda, [-_\lambda,-]'', \varphi'', \psi''´)$ be two BiHom-Poisson conformal algebras. Then,
 $$P=(P'\oplus P'', \circ_\lambda, [ -_\lambda -], \varphi, \psi)$$ is a BiHom-Poisson conformal algebra with :
 \begin{eqnarray}
 \partial &=&\partial'\oplus\partial'',\\
 \varphi=(\varphi''+\varphi'')(x''+x'')&=&\varphi'(x')+\varphi''(x''),\\
 \psi=(\psi'+\psi'')(x'+x'')&=&\psi'(x')+\psi''(x''),\\
 {(x'+x'')\circ_\lambda(y'+y'')}&:=& x'\circ'_\lambda y' +
 y'\circ''_\lambda y'',\\
  {}[(x'+x'')_\lambda(y'+y'')]&:=&[x'_\lambda y']'+[x''_\lambda y'']'',
 \end{eqnarray}
for any $x', y'\in P', x'', y''\in P''$.
\end{cor}

In the rest of this subsection, we consider classical commutativity instead of the BiHom-commutativity. And, we need the following lemmas for the next theorem.
\begin{lem}\label{bhca1}
In any commutative BiHom-associative  superalgebra $(A, \circ, \varphi, \psi)$, we have
\begin{enumerate}
 \item [1)] $\psi^2(a)\circ(\psi(b)\circ\varphi(c))=\psi^2(a)\circ(\psi(b)\circ\psi(c)),$
 \item [2)] $\psi^2(b)\circ(\psi(c)\circ\varphi(a))=(-1)^{(|b|+|c|)|a|}\psi^2(a)\circ(\psi(b)\circ\psi(c)),$
\item [3)]$\psi^2(c)\circ(\psi(a)\circ\varphi(b))=(-1)^{(|a|+|b|)|c|}\psi^2(a)\circ(\psi(b)\circ\psi(c)),$
\end{enumerate}
for any $a, b, c\in\mathcal{H}(A)$.
\end{lem}
\begin{proof}
\begin{enumerate}
 \item [1)] For any $a, b, c\in \mathcal{H}(A)$, we have :
\begin{eqnarray}
 \psi^2(a)\circ(\psi(b)\circ\varphi(c))&=&(-1)^{|a|(|b|+|c|)}(\psi(b)\circ\varphi(c))\circ\psi^2(a)=(-1)^{|a|(|b|+|c|)}\varphi\psi(b)\circ(\varphi(c)\circ\psi(a))\nonumber\\
&=&(-1)^{|a||b|}\varphi\psi(b)\circ(\psi(a)\circ\varphi(c))=(-1)^{|a||b|}(\psi(b)\circ\psi(a))\circ\psi\varphi(c)\nonumber\\
&=&(-1)^{|a||b|+(|a|+|b|)|c|}\varphi\psi(c)\circ(\psi(b)\circ\psi(a))\nonumber\\
&=&(-1)^{|a||b|+(|a|+|b|)|c|}(\psi(c)\circ\psi(b))\circ\psi^2(a)\nonumber\\
&=&\psi^2(a)\circ(\psi(b)\circ\psi(c))\nonumber.
\end{eqnarray}
\item [2)] For any $a, b, c\in \mathcal{H}(A)$, we have :
\begin{eqnarray}
 \psi^2(b)\circ(\psi(c)\circ\varphi(a))
&=&(-1)^{(|a|+|c|)|b|}(\psi(c)\circ\varphi(a))\circ\psi^2(b)
=(-1)^{(|a|+|c|)|b|}\varphi\psi(c)\circ(\varphi(a)\circ\psi(b))\nonumber\\
&=&(-1)^{|b||c|}\varphi\psi(c)\circ(\psi(b)\circ\varphi(a))
=(-1)^{|b||c|}(\psi(c)\circ\psi(b))\circ\psi\varphi(a)\nonumber\\
&=&(-1)^{|b||c|+(|b|+|c|)|a|}\varphi\psi(a)\circ(\psi(c)\circ\psi(b))\nonumber\\
&=&(-1)^{|b||c|+(|b|+|c|)|a|}(\psi(a)\circ\psi(c))\circ\psi^2(b)\nonumber\\
&=&(-1)^{(|a|+|c|)|b|}(\psi(c)\circ\psi(a))\circ\psi^2(b)
=(-1)^{(|a|+|c|)|b|}\varphi\psi(c)\circ(\psi(a)\circ\psi(b))\nonumber\\
&=&(-1)^{|b||c|}\varphi\psi(c)\circ(\psi(b)\circ\psi(a))
=(-1)^{|b||c|}(\psi(c)\circ\psi(b))\circ\psi^2(a)\nonumber\\
&=&(-1)^{(|b|+|c|)|a|}\psi^2(a)\circ(\psi(b)\circ\psi(c))\nonumber\\
&=&(-1)^{(|b|+|c|)|a|}\psi^2(a)\circ(\psi(b)\circ\psi(c))\nonumber.
\end{eqnarray}
 \item [3)] For any $a, b, c\in \mathcal{H}(A)$, we have :
\begin{eqnarray}
 \psi^2(c)\circ(\psi(a)\circ\varphi(b))&=&(-1)^{(|a|+|b|)|c|}(\psi(a)\circ\varphi(b))\circ\psi^2(c)=(-1)^{(|a|+|b|)|c|}\varphi\psi(a)\circ(\varphi(b)\circ\psi(c))\nonumber\\
&=&(-1)^{|a||c|}\varphi\psi(a)\circ(\psi(c)\circ\varphi(b))=(-1)^{|a||c|}(\psi(a)\circ\psi(c))\circ\psi\varphi(b)\nonumber\\
&=&(-1)^{(|a|+|c|)|b|+|a||c|}\varphi\psi(b)\circ(\psi(a)\circ\psi(c))=(-1)^{(|a|+|c|)|b|}\varphi\psi(b)\circ(\psi(c)\circ\psi(a)) \nonumber\\
&=&(-1)^{(|a|+|c|)|b|}(\psi(b)\circ\psi(c))\circ\psi^2(a)\nonumber\\
&=&(-1)^{(|a|+|b|)|c|}\psi^2(a)\circ(\psi(b)\circ\psi(c))\nonumber.
\end{eqnarray}
\end{enumerate}
This achieves the proof.
\end{proof}
\begin{lem}\label{bhca2}
If $(A, \circ, \varphi, \psi)$ is a commutative BiHom-associative superalgebra, we have
\begin{eqnarray}
 \varphi\psi(a)\circ (b\circ c)= (-1)^{|a||b|}\psi(b)\circ(\varphi(a)\circ c).
\end{eqnarray}
\end{lem}
\begin{proof}
For any $a, b, c\in \mathcal{H}(A)$, we have :
\begin{eqnarray}
 \varphi\psi(a)\circ(b\circ c)&=&(-1)^{|a|(|b|+|c|)}(b\circ c)\circ\psi\varphi(a)=(-1)^{|a|(|b|+|c|)}\varphi(b)\circ(c\circ\varphi(a))\nonumber\\
&=&(-1)^{|a||b|}\varphi(b)\circ(\varphi(a)\circ c)
=(-1)^{|a||b|}(b\circ\varphi(a))\circ\psi(c)=(\varphi(a)\circ b)\circ\psi(c)\nonumber\\
&=&\varphi^2(a)\circ(b\circ c)=(-1)^{|b||c|}\varphi^2(a)\circ (c\circ b)\nonumber\\
&=&(-1)^{|b||c|}(\varphi(a)\circ c)\circ\psi(b)\nonumber\\
&=&(-1)^{|a||b|}\psi(b)\circ(\varphi(a)\circ c)\nonumber.
\end{eqnarray}
This ends the proof.
\end{proof}

The below result gives an extension of a  BiHom-Poisson color algebra by a commutative BiHom-associative color algebra.
\begin{thm}\label{tens}
 Let $(A, \circ, \varphi, \psi)$  be a commutative BiHom-associative superalgebra and $(P, \circ'_\lambda, [-_\lambda,-]', \varphi', \psi')$ be a  BiHom-Poisson conformal superalgebra.
Then, the tensor product $A\otimes P$ endowed with the even (bi)linear maps
\begin{eqnarray}
\partial&:&=Id\otimes\partial : A\otimes P\rightarrow A\otimes P,\quad a\otimes x\mapsto a\otimes\partial(x)\nonumber\\
\Phi &:&=\varphi\otimes \varphi' : A\otimes P\rightarrow A\otimes P,\quad a\otimes x\mapsto\varphi(a)\otimes\varphi'(x), \nonumber\\
\Psi &:&=\psi\otimes \psi' : A\otimes P\rightarrow A\otimes P,\quad a\otimes x\mapsto\psi(a)\otimes\psi'(x) ,\nonumber\\
\bullet_\lambda &:& (A\otimes P)\times (A\otimes P)\rightarrow A\otimes P,
\quad (a\otimes x, b\otimes y)\mapsto(-1)^{|b||x|}a\circ b\otimes x\circ'_\lambda y,\nonumber\\
\ldb-_\lambda -\rdb &:& (A\otimes P)\times (A\otimes P)\rightarrow A\otimes P, \quad (a\otimes x, b\otimes y)\mapsto(-1)^{|b||x|}(a\circ b)\otimes [x_\lambda y]',\nonumber
\end{eqnarray}
for all $a, b \in\mathcal{H}(A)$, $x, y\in \mathcal{H}(P)$,
is a  BiHom-Poisson conformal algebra.
\end{thm}
\begin{proof}
It is easy to prove the conformal BiHom-associativity.
For the BiHom-Jacobi conformal identity, we have successively, for any $a, b, c\in \mathcal{H}(A)$ and $x, y, z\in \mathcal{H}(P)$,
\begin{eqnarray}
 &&\hspace{2cm}\ldb\Psi^2(a\otimes x)_\lambda \ldb\Psi(b\otimes y)_\mu \Phi(c\otimes z)\rdb\rdb=\nonumber\\
&&=(-1)^{|y||c|}\ldb(\psi^2(a)\otimes\psi'^2(x))_\lambda (\psi(b)\varphi(c)\otimes[\psi'(y)_\mu \varphi'(z)]')\rdb\nonumber\\
&&=(-1)^{|y||c|+|x|(|b|+|c|)}\psi^2(a)(\psi(b)\varphi(c))\otimes[\psi'^2(x)_\lambda [\psi'(y)_\mu \varphi'(z)]']'\nonumber,
\end{eqnarray}
\begin{eqnarray}
 &&\hspace{2cm}\ldb\Psi^2(b\otimes y)_\lambda \ldb\Psi(c\otimes z)_{-\lambda-\partial} \Phi(a\otimes x)\rdb\rdb=\nonumber\\
&&=(-1)^{|z||a|}\ldb(\psi^2(b)\otimes\psi'^2(y))_\lambda (\psi(c)\varphi(a)\otimes[\psi'(z)_{-\lambda-\partial} \varphi'(x)]')\rdb\nonumber\\
&&=(-1)^{|z||a|+|y|(|c|+|a|)}\psi^2(b)(\psi(c)\varphi(a))\otimes[\psi'^2(y)_\lambda [\psi'(z)_{-\lambda-\partial} \varphi'(x)]']'\nonumber
\end{eqnarray}
and
\begin{eqnarray}
 &&\hspace{2cm}\ldb\Psi^2(c\otimes z)_\lambda \ldb\Psi(a\otimes x)_{-\lambda-\mu\partial} \Phi(b\otimes y)\rdb\rdb=\nonumber\\
&&=(-1)^{|x||b|}\ldb(\psi^2(c)\otimes\psi'^2(z))_\lambda (\psi(a)\varphi(b)\otimes[\psi'(x)_{-\lambda-\mu-\partial} \varphi'(y)]')\rdb\nonumber\\
&&=(-1)^{|x||b|+|z|(|a|+|b|)}\psi^2(c)(\psi(a)\varphi(b))\otimes[\psi'^2(z)_\lambda [\psi'(a)_{-\lambda-\mu-\partial} \varphi'(y)]']'\nonumber,
\end{eqnarray}
The conclusion follows from \ref{bhca1}.

 Now, let us prove the BiHom-Leibniz identity.
\begin{eqnarray}
 &&\hspace{2cm}\ldb\Psi\Phi(a\otimes x)_\lambda (b\otimes y)\bullet_\mu(c\otimes z)\rdb=\nonumber\\
&&=(-1)^{|y||c|}\ldb(\varphi\psi(a)\otimes\psi'\varphi'(x))_\lambda (b\circ c\otimes y\circ_\mu' z)\rdb\nonumber\\
&&=(-1)^{|y||c|+|x|(|b|+|c|)}\varphi\psi(a)\circ(b\circ c)\otimes[\psi'\varphi'(x)_\lambda (y\circ_\mu' z)]'\nonumber\\
&&=(-1)^{|y||c|+|x|(|b|+|c|)}\varphi\psi(a)\circ(b\circ c)\otimes\Big([\psi'(x)_\lambda y]'_{\lambda+\mu}\psi'(z)+(-1)^{|x||y|}\psi'(y)\circ_\mu[\varphi'(x)_\lambda z]')\nonumber\\
&&=(-1)^{|y||c|+|x||b|+|x||c|}\varphi\psi(a)\circ(b\circ c)\otimes[\psi'(x)_\lambda y]'_{\lambda+\mu}\psi'(z)+\nonumber\\
&&\quad+(-1)^{|y||c|+(|x|(|b|+|c|)+|x||y|}\varphi\psi(a)\circ(b\circ c)\otimes\psi'(y)\circ_\mu[\varphi'(x)_\lambda z]'\nonumber.
\end{eqnarray}

By color BiHom-associativity and Lemma \ref{bhca2},
\begin{eqnarray}
 &&\hspace{2cm}\ldb\Psi\Phi(a\otimes x)_\lambda( (b\otimes y)\bullet_\mu(c\otimes z))\rdb\nonumber\\
&&=(-1)^{|x||b|}(\psi(a)\circ b\otimes[\psi'(x)_\lambda y]'_{\lambda+\mu})(\psi(c)\otimes\psi'(z))\nonumber\\
&&\quad+(-1)^{|y||c|+|x|(|b|+|c|)+|x||y|+|a||b|}\psi(b)\circ(\varphi(a)\circ c)\otimes\psi'(y)\circ'_\mu[\varphi'(x)_\lambda z]'\nonumber\\
&&=(-1)^{|x||b|}(\psi(a)\circ b\otimes[\psi'(x), y]'_{\lambda+\mu})(\psi(c)\otimes\psi'(z))\nonumber\\
&&\quad+(-1)^{|x||c|+|x||b|+|x||y|+|a||b|}(-1)^{|a||y|}(\psi(b)\otimes\psi'(y))(\varphi(a)c\otimes[\varphi'(x)_\lambda z]')\nonumber\\
&&=\ldb\psi(a)\otimes\psi'(x)_\lambda b\otimes y\rdb_{\lambda+\mu}(\psi(c)\otimes\psi'(z))\nonumber\\
&&\quad+(-1)^{(|a|+|x|)(|b|+|y|)}(\psi(b)\otimes\psi'(y))\circ_\mu' \ldb(\varphi(a)\otimes\varphi'(x))_\lambda c\otimes z\rdb\nonumber\\
&&=\ldb\Psi(a\otimes x)_\lambda b\otimes y\rdb\circ'_{\lambda+\mu}\Psi(c\otimes z)
+(-1)^{(|a|+|x|)(|b|+|y|)}\Psi(b\otimes y)\circ'_\mu \ldb\Phi(a\otimes x)_\lambda c\otimes z\rdb.\nonumber
\end{eqnarray}
This achieves the proof.
\end{proof}

\begin{cor}
  Let $(A, \circ, \varphi)$  be a commutative Hom-associative
c superalgebra and $(P, \circ'_\lambda, [-_\lambda,-]', \varphi', \psi')$ be a  BiHom-Poisson conformal superalgebra.
Then, the tensor product $A\otimes P$ is a  BiHom-Poisson conformal superalgebra.
\end{cor}
\begin{cor}
  Let $(A, \circ)$  be a commutative associative
 superalgebra and\\ $(P, \circ'_\lambda, [-_\lambda,-]', \varphi', \psi')$ be a  BiHom-Poisson conformal superalgebra.
Then, the tensor product $A\otimes P$ is a  BiHom-Poisson conformal superalgebra.
\end{cor}

{Perspectives :} According to other classical, conformal or Bi(Hom)-types structures, we may give constructions of BiHom-Poisson conformal superalgebras from either derivations, Reynolds operators or element of centroids.

% ------------------------------------------------------------------------
%GATHER{Xbib.bib}   % For Gather Purpose Only
%GATHER{Paper.bbl}  % For Gather Purpose Only

%\bibliographystyle{amsplain}
%\bibliography{reference}

\end{document}